\documentclass[hidelinks,onefignum,onetabnum]{siamart250211}

\usepackage{lipsum}
\usepackage{epstopdf}
\usepackage{algorithmic}
\ifpdf
  \DeclareGraphicsExtensions{.eps,.pdf,.png,.jpg}
\else
  \DeclareGraphicsExtensions{.eps}
\fi

\newsiamremark{remark}{Remark}
\newsiamremark{hypothesis}{Hypothesis}
\crefname{hypothesis}{Hypothesis}{Hypotheses}
\newsiamthm{claim}{Claim}
\newsiamremark{fact}{Fact}
\crefname{fact}{Fact}{Facts}
\usepackage{amscd}
\usepackage{latexsym}
\usepackage{graphicx}
\usepackage{amsmath}
\usepackage{amsfonts}
\usepackage{float}
\usepackage{amssymb}
\usepackage{enumerate}
\usepackage{multirow}
\usepackage{floatflt}
\usepackage{color}
\usepackage{cite}\usepackage{subfigure}
\usepackage{enumitem}
\usepackage{url}
\usepackage{tikz}
\usetikzlibrary{matrix, positioning}
\usepackage{combelow}

\theoremstyle{definition}
\newtheorem{example}{Example}
\theoremstyle{remark}

\newcommand{\R}{\mathbb R}

\newcommand{\bx}{\bar{x}}
\newcommand{\by}{\bar{y}}

\newcommand{\tu}{\tilde{u}}
\newcommand{\tv}{\tilde{v}}

\newcommand{\bv}{\bar{v}}

\newcommand{\bu}{\bar{u}}

\newcommand{\gu}{\texttt{Gurobi}}

\newcommand{\matlab}{\texttt{MATLAB}}

\newcommand{\CC}{\mathcal C}
\newcommand{\RR}{\mathbb R}
\newcommand{\Sim}{{\rm Dis}}
\newcommand{\dis}{{\rm dis}}
\newcommand{\argmax}{\mathop{\rm argmax}}
\newcommand{\argmin}{\mathop{\rm argmin}}

\def \n1{\texttt{np1}}
\def \np2{\texttt{np2}}

\newcommand{\inner}[2]{\langle#1,#2\rangle}

\newcommand{\ind}{\mathbf{i}}

\newcommand{\norm}[1]{\textstyle{\Vert} #1 \textstyle{\Vert}}

\newcommand{\dist}{\mathtt{d}}
\newcommand{\proj}{\mathtt{Proj}}

\newcommand{\Dis}{{\rm Dis}}

\headers{Measuring dissimilarity between convex cones}{W. de Oliveira, V. Sessa, and D. Sossa}

\title{Measuring dissimilarity between convex cones by means of max-min angles\thanks{Submitted to the editors September 2, 2025.
\funding{This work was funded by MATH-AMSUD 23-MATH-09 (AMSUD230018)  MORA-DataS project. D. Sossa was also supported by the Universidad de O'Higgins through grants MOVI2430 and PNTE2502.}}}

\author{Welington de Oliveira\thanks{Mines Paris, Universit\'e PSL, Centre de Math\'ematiques Appliqu\'ees (CMA),  Sophia Antipolis, France 
(\email{welington.oliveira@minesparis.psl.eu}, \email{valentina.sessa@minesparis.psl.eu}). }
\and Valentina Sessa\footnotemark[2]
\and David Sossa\thanks{Universidad de O'Higgins, Instituto de Ciencias de la Ingenier\'ia, Av.\,Libertador Bernardo O'Higgins 611, Rancagua, Chile 
  (\email{david.sossa@uoh.cl}).}
}

\usepackage{amsopn}

\begin{document}

\maketitle

\title{Measuring dissimilarity between convex cones by means of max-min angles}

\author{Welington de Oliveira$^\ast$ \and Valentina Sessa\footnote{ Mines Paris, Universit\'e PSL, Centre de Math\'ematiques Appliqu\'ees (CMA),  Sophia Antipolis, France (welington.oliveira@minesparis.psl.eu, valentina.sessa@minesparis.psl.eu).} \and David Sossa\footnote{Universidad de O'Higgins, Instituto de Ciencias de la Ingenier\'ia, Av.\,Libertador Bernardo O'Higgins 611, Rancagua, Chile (david.sossa@uoh.cl).}    } 

\date{\today}

\maketitle

\begin{abstract}
 This work introduces a novel dissimilarity measure between two convex cones, based on the max-min angle between them. We demonstrate that this measure is closely related to the Pompeiu-Hausdorff distance, a well-established metric for comparing compact sets.
Furthermore, we examine cone configurations where the measure admits simplified or analytic forms.
For the specific case of polyhedral cones, a nonconvex cutting-plane method is deployed to compute, at least approximately, the measure between them. Our approach builds on a tailored version of Kelley's cutting-plane algorithm, which involves solving a challenging master program per iteration.
When this master program is solved locally, our method yields an angle that satisfies certain necessary optimality conditions of the underlying nonconvex optimization problem yielding the dissimilarity measure between the cones. 
As an application of the proposed mathematical and algorithmic framework, we address the image-set classification task under limited data conditions, a task that falls within the scope of the \emph{Few-Shot Learning} paradigm. In this context, image sets belonging to the same class are modeled as polyhedral cones, and our dissimilarity measure proves useful for understanding whether two image sets belong to the same class.
\end{abstract}

\begin{keywords}
    convex cone, max-min angle, projection, Pompeiu-Hausdorff distance, cutting-plane, image-set classification
\end{keywords}

\begin{MSCcodes}
    52A40, 90C26, 90C46, 90C47, 49M37
\end{MSCcodes}

\section{Introduction}
Measuring the dissimilarity between two mathematical objects is a fundamental problem in mathematics and its applications. For instance, in data classification problems, one should have a mechanism to measure how alike two data objects are. In particular, in image-set classification problems, it is common to model a data image-set as an element of a Grassmann manifold (collection of all linear subspaces of fixed dimension). Thus, the problem of measuring the dissimilarity between two data image-sets can be approached by computing the distance between two linear subspaces. There are many ways of defining a distance on a Grassmann manifold, most of them are based on the principal angles between subspaces, like the Grassmann distance, projection distance, spectral distance, etc., see \cite{YeLim}. Principal angles are fundamental geometric tools for comparing two subspaces. Moreover, they can be efficiently computed since it reduces to the computation of the singular values of some matrix associated with the subspaces \cite{Miao}.

As pointed out in \cite{IS2010}, a fruitful extension of the concept of linear subspace is that of convex cone. Convex cones also play an ubiquitous role in many branches of applied mathematics. For instance, in \cite{Sogi}, it has been reported that some classes of data image-sets can be more accurately modeled by convex cones rather than linear subspaces. Thus, measuring the dissimilarity between convex cones is a relevant issue in classification tasks, for instance. For this purpose, many distances between convex cones are known in the literature. We refer to the work of Iusem and Seeger \cite{IS2010} for a survey on this topic.

Let $\CC_n$ denote the set of nonempty closed convex cones in $\RR^n$, and let $S_n:=\{u\in\mathbb R^n:\Vert u\Vert =1\}$ be the unit sphere in $\RR^n$. As happens with distances between linear subspaces, it is natural to ask whether there are distances on $\CC_n$ based on a concept of angles between cones. Iusem and Seeger \cite{IS1,IS2,IS3,IS4} have initiated the study of maximal angle on a cone. Later, Seeger and Sossa \cite{SeSo1,SeSo2} have extended this study for two cones as follows: Given $P,Q\in\CC_n$, the maximal angle between $P$ and $Q$ corresponds to the optimal value of the problem
\begin{equation}\label{maxangle}
\max_{u\in P\cap S_n,\,v\in Q\cap S_n}\arccos\,\langle u,v\rangle.
\end{equation}
Furthermore, the stationary values of the above problem give us the concept of critical angles between $P$ and $Q$, which is an extension of that of principal angles between subspaces. Indeed, when $P$ and $Q$ are linear subspaces, both concepts coincide. However, it is not clear to us how to define a distance from critical angles since there are some computing limitations for finding all critical angles, and because problem \eqref{maxangle} is NP-hard, see \cite{dOSS,BGS}. Even more, only the maximal angle will not be enough  to define a distance on $\CC_n$ since, in general, the maximal angle between two identical cones is nonzero. If we consider the minimization counterpart of \eqref{maxangle}, we get the concept of minimal angle between cones, but different cones with nonempty intersection have a minimal angle equal to zero; so, the minimal angle does not define a distance either.

To overcome the limitations of the maximal and minimal angles between cones for defining a distance on $\CC_n$, we introduce a new concept of angle: for $P,Q\in\CC_n$, the max-min angle between $P$ and $Q$ is defined as 
\begin{equation}\label{maxminPQ}
\Theta(P,Q):=\max_{u\in P\cap S_n}\min_{v\in Q\cap S_n}\arccos\,\inner{u}{v}.
\end{equation}

We shall see that $\Theta(P,Q)$ and $\Theta(Q,P)$ capture distinct dissimilarity information between $P$ and $Q$. Moreover, if $\Theta(P,Q)=\Theta(Q,P)=0$ then $P=Q$. Thus,  we shall prove that the following is a distance on $\CC_n$:
\[
\Sim_r(P,Q)=2\left\Vert\left(\sin\left[\frac{\Theta(P,Q)}{2}\right],\,\sin\left[\frac{\Theta(Q,P)}{2}\right]\right)\right\Vert_r,
\]
where $\Vert\cdot\Vert_r$ is the $\ell^r$-norm in $\RR^2$, for $1\leq r\leq\infty$. Our construction of such a distance is inspired by the Pompeiu-Hausdorff distance between compact sets.

Given $P,Q \in \CC_n$, for a practical computation of the max-min angle $\Theta(P,Q)$ it is convenient to compute $\cos[\Theta(P,Q)]$ which is the optimal value of the min-max problem:
\begin{equation}\label{eq:minmaxPQ}
\min_{u\in P\cap S_n}\max_{v\in Q\cap S_n}\langle u,v\rangle.
\end{equation}
Indeed, the above problem reduces to 
\[\min_{u\in P\cap S_n}F_Q(u) \quad \mbox{where} \quad F_Q(u):=\max_{v\in Q\cap S_n}\langle u, v\rangle
\]
is convex since it is the support function of $Q\cap S_n$. We show in Proposition~\ref{prop FQ} below that $F_Q(u)$ is related to the projection of $u$ onto the intersection of a cone and a sphere, whose computation can be performed by means of the tools developed by Bauschke, Bui and Wang \cite{Bauschke}. 
We shall see that the minimum of $F_Q$ over $P\cap S_n$ can be explicitly computed for some particular instances of cones like linear subspaces and revolution cones. For the more involving case of polyhedral cones, we tackle the problem with a different approach. A specialized nonconvex version of the well-known cutting-plane method~\cite{Kelley_1960} is implemented to globally solve the min-max problem~\eqref{eq:minmaxPQ}.
Our algorithmic framework
involves solving a challenging
master program per iteration: a nonconvex quadratically constrained optimization problem with a linear objective. When the master program is solved locally to generate iterates (as is typically required in the large-scale problems arising from image-set classification) our method produces an angle that satisfies necessary optimality conditions of the underlying problem~\eqref{eq:minmaxPQ}.

To report the performance of our proposed mathematical and algorithmic framework, we address the
image-set classification task under limited data conditions. In this context, image sets belonging to the same class are modeled as
polyhedral cones. Indeed, images belonging to the same class, referred to as class $i$, can be seen as generators of a (polyhedral) cone $Q_i$. Different classes produce distinct cones that are constructed in a training phase. When presented with a new set of images that generates a cone \( P \), the goal in the test phase is to classify this new set into one of the existing classes.
One effective method for performing this classification is by comparing the generated cones. To accomplish this, we propose using the distance $\Sim_r(P,Q)$. We classify the new set of images into class $ i^*$, which corresponds to the smallest distance. This can be expressed formally as $ i^* \in \arg\min_{i}\, \Sim_r(P, Q_i)$. 
Our numerical results show that the proposed dissimilarity measure allows for accurate classification of the image set in the ETH-80 dataset~\cite{Chen2020}. We achieved an average $90 \%$ accuracy in classification, even when only a few images (limited data) were used to build the reference cones (in the training phase).

We note that image-set classification under limited data conditions falls within the scope of \emph{Few-Shot Learning} (FSL), a framework designed to mimic the human ability to learn from few examples \cite{FSL}. Unlike traditional supervised learning, which relies on large labeled datasets and extensive training, FSL is particularly valuable in domains where data is scarce, costly to label, or simply impractical in situations such as rare diseases, unique handwriting styles, or newly discovered species. In this context, providing a robust mathematical and algorithmic framework for comparing cones can significantly advance FSL applications.

The remainder of this work is organized as follows. Section~\ref{sec:dist} introduces the proposed dissimilarity measure between convex cones based on the concept of the max-min angle, and establishes its connection to the Pompeiu-Hausdorff distance. Section~\ref{sec:comp-dist} addresses the problem of computing this measure, focusing on the underlying optimization problems and providing a thorough analysis of their mathematical properties and optimality conditions. Special cases and numerical strategies for computing the dissimilarity are discussed in Section~\ref{sec:cases}. Section~\ref{sec:poly} is dedicated to the case of polyhedral cones, where we present an algorithmic framework based on a cutting-plane method. Section~\ref{sec:numexp} reports numerical experiments using two variants of the proposed algorithm, applied to randomly generated polyhedral cones and an image-set classification task. Finally, Section~\ref{sec:conclusion} concludes the work with remarks and directions for future research.

\paragraph{Notation} Given $P \in \CC_n$, its polar cone is denoted by $P^\circ=\{w:\, \inner{w}{u}\leq 0\; \mbox{for all } u \in P\}$ and its dual is $P^* = -P^\circ$. We also denote by $P^\perp$ the set $\{w\in \R^n: \inner{w}{u}=0\; \mbox{for all } u \in P\}$. Given a closed set $X \subset \R^n$ and $u \in \R^n$, we denote  $\proj_{X}(u):=\arg\min _{x \in X} \norm{x-u}$, $\dist_{X}(u):=\min _{x \in X} \norm{x-u}$, and $\dist^2_{X}(u):=\min _{x \in X} \norm{x-u}^2$. Furthermore, $\mathtt{cl\, conv}\,X$ denotes the closed convex hull of $X$, and $\mathtt{span}\,X$ denotes the linear span of $X$. When $X$ is a linear space, $\mathtt{dim}\,X$ denotes its dimension. The unit sphere and the unit ball in $\R^n$ are denoted, respectively, by $S_n:=\{u\in\R^n:\Vert u\Vert= 1\}$ and $B_n:=\{u\in\R^n:\Vert u\Vert\leq 1\}$.

\section{Distance between convex cones}\label{sec:dist}
We start this section by proving some elementary properties of $\Theta(P,Q)$ defined by problem~\eqref{maxminPQ}.
\subsection{Overview of the max-min angle between two cones}
The following result establishes the well-definiteness of the max-min angle.
\begin{proposition}\label{pr:elementary}
Let $P,Q\in\CC_n$. The following statements are satisfied:
\begin{enumerate}
\item[(a)] $\Theta(P,Q)$ is well-defined. That is, the max-min problem \eqref{maxminPQ} has a solution.
\item[(b)] $\Theta(P,Q)$ is invariant under orthogonal linear transformations. That is, for every $n$-by-$n$ orthogonal matrix $V$, $\Theta(V(P),V(Q))=\Theta(P,Q)$.
\end{enumerate}
\end{proposition}
\begin{proof}
(a). Since the function $\arccos(\cdot)$ is decreasing, the cosine of the optimal value of problem \eqref{maxminPQ}  coincides with the optimal value of the following min-max problem:
\[\cos[\Theta(P,Q)]=\min_{u\in P\cap S_n}\max_{v\in Q\cap S_n}\langle u,v\rangle.\]
The above problem has at least one solution since $P\cap S_n$ and $Q\cap S_n$ are compact sets, and $v\mapsto\langle u,v\rangle$ and $u\mapsto\max_{v\in Q\cap S_n}\langle u,v\rangle$ are continuous (convex) functions.

(b). The equality $\Theta(V(P),V(Q))=\Theta(P,Q)$ is true because 
\[\inner{Vx}{Vy} = \inner{x}{V^\top Vy}  = \inner{x}{y}.\]
Indeed,
\begin{eqnarray*}
\max_{u\in V(P)\cap S_n}\min_{v\in V(Q)\cap S_n}\arccos\,\inner{u}{v}&=&
\max_{x\in P\cap S_n}\min_{y\in Q\cap S_n}\arccos\,\inner{Vx}{Vy}\\
&=&\max_{x\in P\cap S_n}\min_{y\in Q\cap S_n}\arccos\,\inner{x}{y}.
\end{eqnarray*}
It concludes the proof.
\end{proof}
In order to show the intuition behind the fact that $\Theta(P,Q)$ captures information about the dissimilarity between $P$ and $Q$, we exemplify the max-min problem~\eqref{maxminPQ} with bi-dimensional examples. Fig.~\ref{fig:Example}(a) depicts two cones $P$ (in red) and $Q$ (in green).
For any given $u \in P\cap S_n$, the inner minimization problem in~\eqref{maxminPQ} seeks to find a vector $v_u \in Q\cap S_n$ that minimizes the angle with $u$. For the example of Fig.~\ref{fig:Example}(a), such a $v_u$ always lies in the superior boundary of $Q$ regardless of the given $u \in P\cap S_n$. The outer maximization problem in~\eqref{maxminPQ} then seeks to find $\bar u \in P\cap S_n$ yielding the maximum of minimal angles $\arccos\, \inner{u}{v_u}$. Such an $\bar u$ lies in the superior boundary of $P$. The max-min angle $\Theta(P,Q)=\arccos\, \inner{\bar u}{\bar v}$, with $\bar v = v_{\bar u}$, is thus the optimal value of~\eqref{maxminPQ}. 
A similar analysis by switching the roles of $P$ and $Q$ gives the angle $\Theta(Q,P)$ marked in blue in Fig.~\ref{fig:Example}(a). Clearly, such angles are distinct. Fig.~\ref{fig:Example}(b) shows an example where $\Theta(P,Q)=0$ whereas $\Theta(Q,P)>0$. Hence, looking only at a single max-min angle $\Theta(P,Q)$ is not enough to assert that $P$ and $Q$ in  Fig.~\ref{fig:Example}(b) are different cones. It is now intuitive that $\Theta(P,Q)=\Theta(Q,P)=0$ implies that $P$ and $Q$ are equal, and vice-versa. In Theorem~\ref{th:dist}, we prove not only that this claim holds but also that these angles induce a distance on $\CC_n$. But before that, let us briefly comment on the inversion of the maximization and minimization roles in problem~\eqref{maxminPQ}.
\begin{figure}[htb]
\centering 
\subfigure[]
{\includegraphics[width=1.5in]{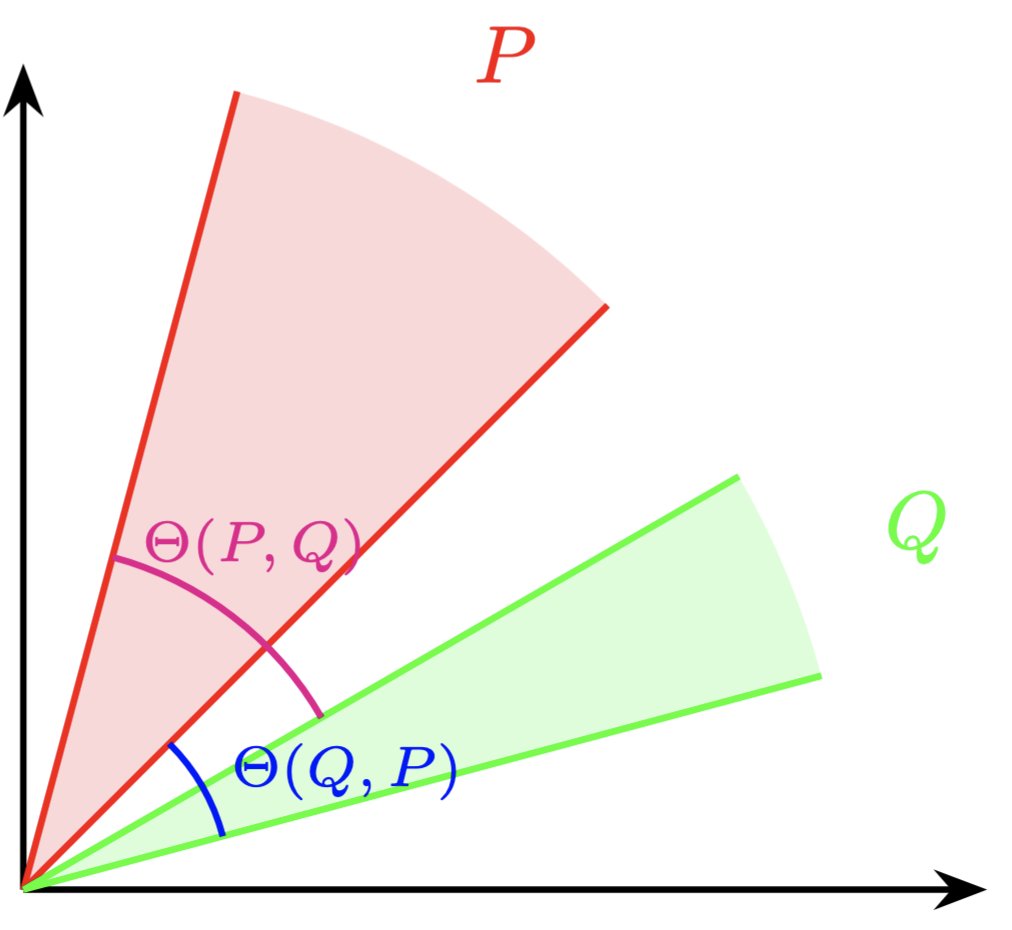}}
\subfigure[]
{\includegraphics[width=1.5in]{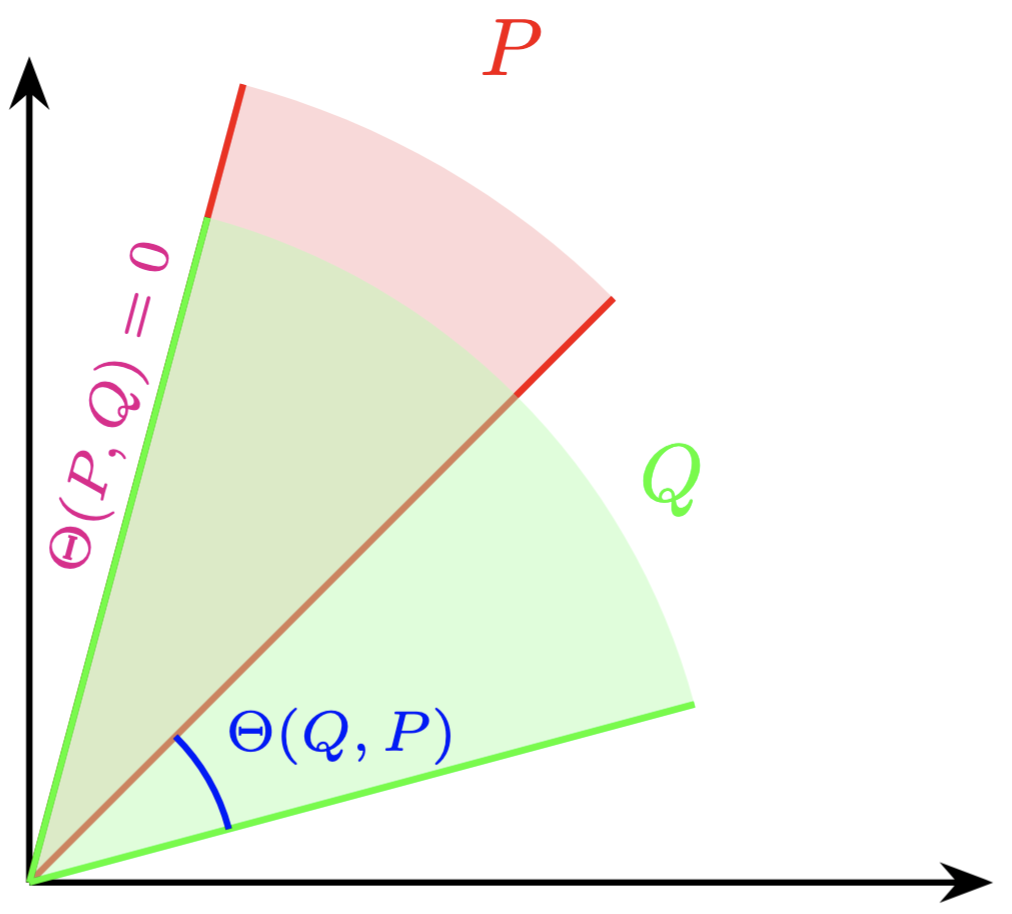}}
\subfigure[]
{\includegraphics[width=1.5in]{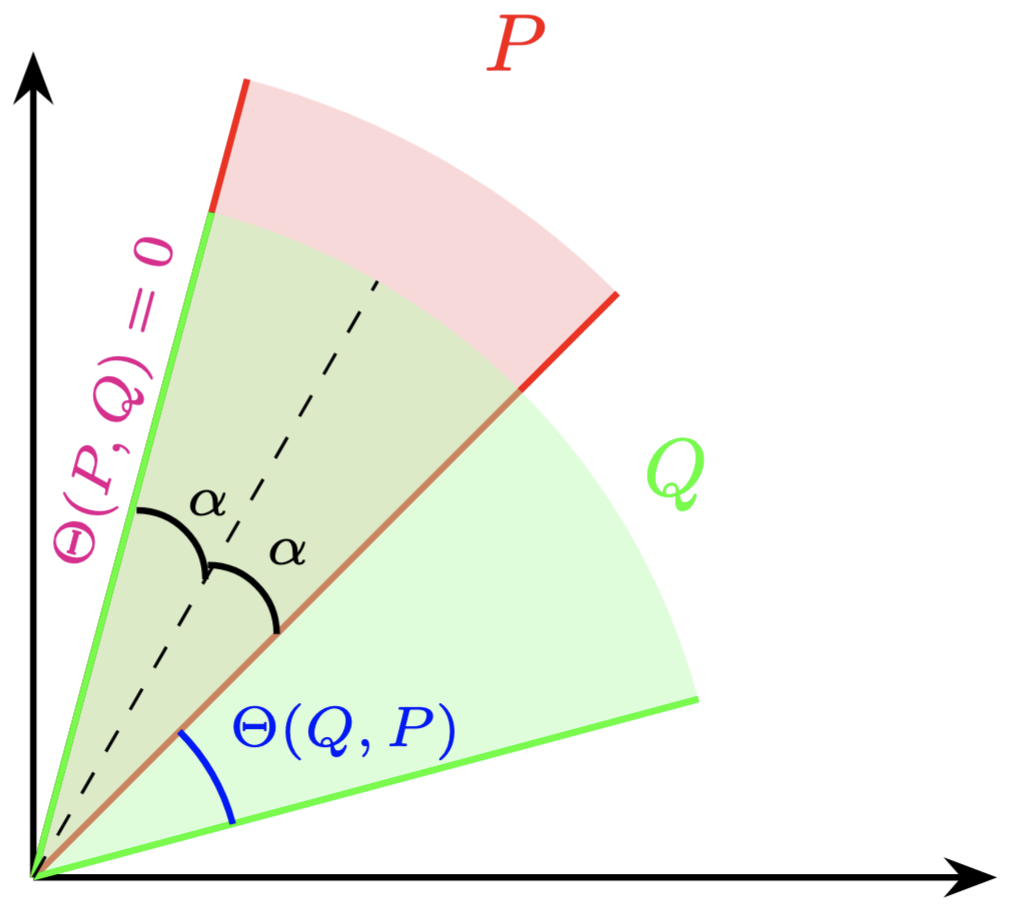}}
 \caption{Bi-dimensional illustration of the max-min problem~\eqref{maxminPQ}.
 }
  \label{fig:Example}
\end{figure}

Let us consider the min-max counterpart of problem \eqref{maxminPQ}. That is,
\begin{equation}\label{minmaxPQ}
\widehat\Theta(P,Q):=\min_{u\in P\cap S_n}\max_{v\in Q\cap S_n}\arccos\,\inner{u}{v}.
\end{equation}
We point out that there is no loss of generality in focusing attention just on the max-min problem \eqref{maxminPQ}.  Indeed, the next proposition shows that $\Theta(P,Q)$ and $\widehat\Theta(P,-Q)$ are supplementary angles.

\begin{proposition}\label{equal} For any $P,Q\in\CC_n$, 
$
\Theta(P,Q)+\widehat\Theta(P,-Q)=\pi.
$
\end{proposition}
\begin{proof}
    Observe that $\arccos(-\langle u,v\rangle)=\pi-\arccos\langle u,v\rangle$. Then,
    \begin{eqnarray*}
    \widehat\Theta(P,-Q)&=&\min_{u\in P\cap S_n}\max_{v\in Q\cap S_n}\arccos\,(-\inner{u}{v})\\
    &=&\pi-\max_{u\in P\cap S_n}\min_{v\in Q\cap S_n}\arccos\,\inner{u}{v}=\pi-\Theta(P,Q),
    \end{eqnarray*}
    which is equivalent to $\Theta(P,Q)+\widehat\Theta(P,-Q)=\pi.$
\end{proof}

For min-max problems, an important result is the celebrated minimax theorem~\cite{Simons95}, allowing us to invert, under convexity and concavity assumptions, the order of maximization and minimization in the underlying problem. As problem~\eqref{maxminPQ} does not satisfy such assumptions, inverting the order of the maximization and minimization gives an upper bound. Indeed, the following inequality holds in generality:
\begin{equation}\label{ineq}
\begin{array}{lcl}
\Theta(P,Q)&=&\displaystyle \max_{u\in P\cap S_n}\min_{v\in Q\cap S_n}\arccos\,\inner{u}{v}\\
 &\leq& 
\displaystyle \min_{v\in Q\cap S_n}\max_{u\in P\cap S_n}\arccos\,\inner{u}{v}=\widehat\Theta(Q,P).
\end{array}
\end{equation}
The angle computed by the right-hand side problem above is illustrated by $\alpha$ in Fig.~\ref{fig:Example}(c). This is an example where $\Theta(P,Q)$ is strictly inferior to the optimal value of the right-hand side problem above. Below, we elaborate on this example with more precision. 

\begin{example}\label{ex:strict}
Let $P$ and $Q$ be the closed convex cones in $\R^2$ illustrated in Fig.~\ref{fig:Example}(c). Observe that these cones are revolution cones (see Section\,\ref{Sec:revolution}). Indeed, $P$ is defined by fixing the revolution (central) axis $b_1\in S_2$ and the half-aperture angle $\phi_1$; that is, $P={\rm Rev}(\phi_1,b_1)$. Analogously, $Q={\rm Rev}(\phi_2,b_2)$. The precise elements of these cones are
\[\phi_1=\pi/12,\;b_1=(\cos(\pi/3),\sin(\pi/3)),\quad \phi_2=\pi/6,\;b_2=(\cos(\pi/4),\sin(\pi/4)).\]
Observe that $\arccos\,\langle b_1,b_2\rangle=\pi/12$. On the other hand, $-P={\rm Rev}(\phi_1,-b_1)$ and $\arccos\,\langle -b_1,b_2\rangle=(11/12)\pi$. Thus, by using Theorem\,\ref{th:rev} (below) and Proposition\,\ref{equal} we get
\[\Theta(P,Q)=0,\quad\widehat\Theta(Q,P)=\pi-\Theta(Q,-P)=\frac{\pi}{12}.\]
Thus, the inequality in \eqref{ineq} holds strictly.
\end{example}
\subsection{Distances between compact sets and between convex cones}
Let $\mathcal K_n$ denote the set of nonempty compact subsets of $\R^n$.
The construction of our metrics on $\CC_n$ relies on metrics on  $\mathcal K_n$. Recall that a classic distance between $C$ and $D$ in $\mathcal K_n$ is the Pompeiu-Hausdorff distance, which is defined as
\begin{equation}\label{hauss}
{\rm Hauss}(C,D):=\max\left\{\max_{x\in C}\dist_D(x),\,\max_{x\in D}\dist_C(x)\right\},
\end{equation}
where $\dist_C(x):=\min_{y\in C}\Vert x-y\Vert$ is the distance of $x\in\RR^n$ to $C$, see \cite[Chapter 4]{Rock-Wets}. Observe that the Pompeiu-Hausdorff distance is composed of the term 
\[\Lambda(C,D):=\max_{x\in C}\dist_D(x).\]
It is known that $\Lambda$ satisfies the axioms of a distance, except symmetry, see \cite{Danet,Sendov}. This is explained in the next known proposition. We give the proof for the sake of completeness. 
\begin{proposition}\label{quasi}
Let $C,D,E\in\mathcal K_n$. The map $\Lambda:\mathcal K_n\times\mathcal K_n\to\RR_+$ satisfies the following conditions:
\begin{enumerate}
\item[(a)] $\Lambda(C,C)=0$;
\item[(b)] if $\Lambda(C,D)=\Lambda(D,C)=0$, then $C=D$;
\item[(c)] $\Lambda(C,D)\leq \Lambda(C,E)+\Lambda(E,D)$.
\end{enumerate}
\end{proposition}
\begin{proof}
Item (a). $\Lambda(C,C)=0$ because $\dist_C(x)=0$ for all $x\in C$.

Item (b). Let us prove the contrapositive statement. Suppose that $C\neq D$, then $C\not\subset D$ or $D\not\subset C$. If $C\not\subset D$, then there is $\bx\in C$ such that $\bx\notin D$. Therefore, $\dist_D(\bx)>0$, which implies that $\Lambda(C,D)\geq\dist_D(\bx)>0$. Analogously, if $D\not\subset C$ then $\Lambda(D,C)>0$.

Item (c). Let $x\in C$, and let $k(x)\in\argmin_{k\in E}\Vert x-k\Vert$. Then, from the triangle inequality of the Euclidean norm, we have that for all $y\in D$,
\[\Vert x-y\Vert\leq \Vert x-k(x)\Vert+\Vert k(x)-y\Vert=\min_{k\in E}\Vert x-k\Vert + \Vert k(x)-y\Vert=\dist_E(x)+ \Vert k(x)-y\Vert.\]
Then, the above inequality is preserved if we take the minimum on $D$:
\[\dist_D(x)=\min_{y\in D}\Vert x-y\Vert\leq \dist_E(x) + \min_{y\in D}\Vert k(x)-y\Vert=\dist_E(x)+\dist_D(k(x)).\]
Since $x\in C$ is arbitrary, we have that
\begin{equation}\label{almost}
\max_{x\in C}\dist_D(x)\leq \max_{x\in C}\dist_E(x) + \max_{x\in C}\dist_D(k(x)).
\end{equation}
Furthermore, it is clear that $\{k(x):x\in C\}\subseteq E$. Then, $\max_{x\in C}\dist_D(k(x))\leq \max_{k\in E}\dist_D(k)$.
By transitivity of this inequality with \eqref{almost},  we conclude:
\[\max_{x\in C}\dist_D(x)\leq \max_{x\in C}\dist_E(x) + \max_{k\in E}\dist_D(k).\]
Consequently, the desired triangle inequality holds.
\end{proof}
There are other metrics that can be constructed in $\mathcal K_n$ by using $\Lambda$. For instance, the Pompeiu-Eggleston distance \cite[Chapter\,9]{Deza} is defined as
\begin{equation}\label{egg}
{\rm Egg}(C,D):=\Lambda(C,D)+\Lambda(D,C).
\end{equation}
We notice that the metrics \eqref{hauss} and \eqref{egg} are obtained by taking the $\ell^\infty$-norm and $\ell^1$-norm to the vector $(\Lambda(C,D),\,\Lambda(D,C))\in\RR^2$. The following proposition shows that we can define a distance on $\mathcal K_n$ by considering any $\ell^r$-norm, for $1\leq r\leq \infty$, on $\R^2$.
\begin{proposition}\label{prop:disr}
For $1\leq r\leq \infty$, let ${\rm dis}_r:\mathcal K_n\times\mathcal K_n\to\RR^n_+$ defined by
\begin{equation}\label{dissset}
{\rm dis}_r(C,D):=\Vert(\Lambda(C,D),\,\Lambda(D,C))\Vert_r.
\end{equation}
Then, ${\rm dis}_r$ is a distance on $\mathcal K_n$. That is, for all $C,D,E\in\mathcal K_n$, the following axioms are satisfied:
\begin{enumerate}
\item[(a)] $\dis_r(C,D)\geq 0$;
\item[(b)] $\dis_r(C,D)=\dis_r(D,C)$;
\item[(c)] $\dis_r(C,D)=0$ if and only if $C=D$;
\item[(d)] $\dis_r(C,D)\leq \dis_r(C,E)+\dis_r(E,D).$
\end{enumerate}
\end{proposition}
\begin{proof}
Items (a) and (b) follow directly from the definition \eqref{dissset}. 

Let us prove (c). If $\dis_r(C,D)=0$ then $\Lambda(C,D)=0$ and $\Lambda(D,C)=0$. Then, from Proposition\,\ref{quasi}(b) we get $C=D$. Conversely, if $C=D$ then $\Lambda(C,D)=\Lambda(D,C)=0$ because of Proposition\,\ref{quasi}(a). It follows that $\dis_r(C,D)=0$.

Let us prove the triangle inequality (d). Let $C,D,E\in\mathcal K_n$. We know from Proposition\,\ref{quasi}(c) that $\Lambda$ satisfies the triangle inequality. Then,
\begin{equation}\label{2ineq}
\Lambda(C,D)\leq \Lambda(C,E)+\Lambda(E,D)\quad\text{and}\quad \Lambda(D,C)\leq \Lambda(D,E)+\Lambda(E,C).
\end{equation}
Hence,
\begin{eqnarray}
\dis_r(C,E)+\dis_r(E,D)&=&\Vert(\Lambda(C,E),\,\Lambda(E,C))\Vert_r + \Vert(\Lambda(E,D),\,\Lambda(D,E))\Vert_r\nonumber\\
&\geq& \Vert(\Lambda(C,E)+\Lambda(E,D),\,\Lambda(E,C)+\Lambda(D,E))\Vert_r,\label{ineqtri}
\end{eqnarray}
where the inequality is because of the triangular inequality of $\Vert\cdot\Vert_r$. Observe that $\Vert\cdot\Vert_r$ is isotonic in $\R^2_+$. It means, if $x,y\in\R^2_+$ are such that $x_1\leq y_1$ and $x_2\leq y_2$, then $\Vert x\Vert_r\leq\Vert y\Vert_r$. Thus, from \eqref{2ineq} we deduce
\begin{equation}\label{ineqtri2}
\begin{array}{rcl}
\dis_r(C,D)&=&\Vert(\Lambda(C,D),\,\Lambda(D,C))\Vert_r\\&\leq&\Vert (\Lambda(C,E)+\Lambda(E,D),\,\Lambda(D,E)+\Lambda(E,C))\Vert_r.
\end{array}
\end{equation}
Therefore, by combining \eqref{ineqtri} and \eqref{ineqtri2} we get the desired triangle inequality.
\end{proof}

Now, we explain how to construct a distance on $\CC_n$ (the set of closed convex cones in $\R^n$). Since the elements in $\CC_n$ are not bounded, we can not use the distance $\dis_r$ directly because it is for compact sets. However, we can truncate the elements in $\CC_n$ by intersecting them with the unit sphere of $\R^n$. Then, we can apply $\dis_r$ on these truncated sets and it will induce a distance on $\CC_n$. A similar procedure was also applied in \cite{IS2010} and \cite{SeSo1}. This is specified in the next theorem. We also show the connection of such a distance with the max-min angles $\Theta(P,Q)$ and $\Theta(Q,P)$. 
\begin{theorem}\label{th:dist}
For $1\leq r\leq\infty$, let $\Sim_r:\CC_n\times\CC_n\to\RR_+$ be defined as
\begin{equation}\label{simr}
\Sim_r(P,Q):={\rm dis}_r(P\cap S_n,Q\cap S_n).
\end{equation}
Then, $\Sim_r$ is a distance on $\mathcal C_n$. Furthermore, it admits the formula
\begin{equation}\label{DisrTheta}
\Sim_r(P,Q)=2\left\Vert\left(\sin\left[\frac{\Theta(P,Q)}{2}\right],\,\sin\left[\frac{\Theta(Q,P)}{2}\right]\right)\right\Vert_r.
\end{equation}
\end{theorem}
\begin{proof}
That $\Dis_r$ is a distance on $\CC_n$ follows directly from Proposition\,\ref{prop:disr}. Let us prove the formula \eqref{DisrTheta}. By the definitions of $\Lambda(P\cap S_n,Q\cap S_n)$ and $\Theta(P,Q)$, and by using the identity $1-\cos\alpha=2\sin^2(\alpha/2)$ we get,
\begin{eqnarray*}
\Lambda(P\cap S_n,Q\cap S_n)&=&\max_{u\in P\cap S_n}\min_{v\in Q\cap S_n}\Vert u-v\Vert=\left( \max_{u\in P\cap S_n}\min_{v\in Q\cap S_n}\Vert u-v\Vert^2\right)^{1/2}\\
&=&\sqrt{2}\left(1-\min_{u\in P\cap S_n}\max_{v\in Q\cap S_n}\langle u,v\rangle\right)^{1/2}\\
&=&\sqrt{2}\left(1-\cos[\Theta(P,Q)]\right)^{1/2}=2\sin[\Theta(P,Q)/2].
\end{eqnarray*}
Thus, the desired formula follows from $\eqref{simr}$ and \eqref{dissset}.
\end{proof}

\section{On computing the dissimilarity distance}\label{sec:comp-dist}
In the previous section, we have seen that the max-min problem \eqref{maxminPQ} is equivalent, in terms of solution, to the following min-max problem:
\begin{equation}\label{minimax}
\cos \left[\Theta(P,Q)\right]=\min_{u\in P\cap S_n}\max_{v\in Q\cap S_n}\langle u,v\rangle.
\end{equation}
For $u\in\mathbb R^n$, we set
\begin{equation}\label{support}
F_Q(u):=\max_{v\in Q\cap S_n}\langle u,v\rangle,
\end{equation}
which corresponds to the support function of $Q\cap S_n$. Hence, 
problem \eqref{minimax} can be formulated as minimizing the support function of $Q\cap S_n$ over $P\cap S_n$:
\begin{equation}\label{minsupport}
\mathfrak{s}(P,Q):=\min_{u\in P\cap S_n}F_Q(u).
\end{equation}
Thus, our original problem \eqref{maxminPQ} reduces to solving problem \eqref{minsupport}. Indeed, we have 
\[\Theta(P,Q)=\arccos\left(\mathfrak{s}(P,Q)\right).\]
From now on, we will focus on solving problem \eqref{minsupport}. We start by listing some properties of the support function $F_Q$ that will be used later. 
\begin{proposition}\label{prop FQ}
Let $Q\in\mathcal C_n$, and let $u\in\R^n$. The following statements hold:
\begin{enumerate}
\item[(a)] The function $F_Q$ given in \eqref{support} is convex and its domain is $\R^n$.
\item[(b)]  $\displaystyle F_Q(u)=\frac{1+\norm{u}^2-\dist_{Q\cap S_n}^2(u)}{2}$.
\item[(c)] $\proj_{Q\cap S_n}(u)=\arg\max_{v \in Q\cap S_n} \inner{u}{v}$.
\item[(d)] $\partial F_Q(u) = \mathtt{conv}\, \proj_{Q\cap S_n}(u)$.
\item[(e)] $F_Q(u) = \inner{u}{v'}$ for all $v'\in \proj_{Q\cap S_n}(u)$.
\item[(f)] If $u \notin Q^\circ$, then $\partial F_Q(u)=\{\proj_Q(u)/\Vert \proj_Q(u)\Vert\}$. 
\end{enumerate}
\end{proposition}
\begin{proof}
Convexity of $F_Q$ is trivial.
As $Q\cap S_n$ is a compact set, $F_Q$ is real-valued. This proves item (a).
    Observe that for all $v \in Q\cap S_n$, we have that $\norm{v-u}^2=1+\norm{u}^2-2\inner{u}{v}$. Thus, $\dist_{Q\cap S_n}^2(u) = \min_{v\in Q\cap S_n}\norm{v-u}^2=1+\norm{u}^2-2\max_{v\in Q\cap S_n} \inner{u}{v} = 1+\norm{u}^2-2F_Q(u)$, giving item (b).
This last development also shows  that $\proj_{Q\cap S_n}(u) =\arg\max_{v\in Q\cap S_n}\langle u,v\rangle$, proving items (c) and (e). 
Let us prove item (d).
It follows from \cite[Eq. 5]{Hantoute_Lopez_2008} that $\partial F_Q(u) = \mathtt{cl\, conv}\, \arg\max_{Q\cap S_n} \; \langle u, v \rangle$.  Since $Q \cap S_n$ is compact,  we can drop the closure: $\partial F_Q(u)=\mathtt{conv}\, \arg\max_{Q\cap S_n} \; \langle u, v \rangle= \mathtt{conv}\,\proj_{Q\cap S_n}(u)$. 
Item (f) is proved in \cite[Thm 8.1]{Bauschke}. 
\end{proof}
As a consequence of item (a) of Proposition\,\ref{prop FQ}, we deduce that $F_Q$ is continuous on $\R^n$. From this fact and the compactness of $P\cap S_n$, the extreme value theorem ensures that \eqref{minsupport} is always solvable.

We now provide a necessary optimality condition for a point $\bar u\in \R^n$ to solve problem~\eqref{minsupport}.
\begin{proposition} \label{prop:stat}
    For any stationary point $\bar u$ to problem~\eqref{minsupport} there exists $\bv\in Q$ such that the following system is satisfied:
\begin{equation}\label{optim}
 \begin{array}{l}
P\ni\bu\perp \bv -\langle\bu,\bv\rangle \bu\in P^\ast,\\
\norm{\bar u}=1,\; \bar v \in \mathtt{conv\,}\arg\max_{v \in Q\cap S_n} \inner{\bar u}{v}.
\end{array}   
\end{equation}
Furthermore, if $\arg\max_{v \in Q\cap S_n} \inner{\bar u}{v}$ is a singleton, then the above system is equivalent to 
\begin{equation}\label{optim2}
 \begin{array}{l}
P\ni\bu\perp \bv -\langle\bu,\bv\rangle \bu\in P^\ast,\\
Q\ni\bv\perp \bu -\langle\bu,\bv\rangle \bv\in Q^\circ,\\
\norm{\bar u}=\norm{\bar v} =1.
\end{array}
\end{equation}
This is the case, for instance, if $\bar u\not \in Q^\circ$.
\end{proposition}
\begin{proof}
 Observe that \eqref{minsupport} can be rewritten as
\[
\min_{u } \; F_Q(u) + \ind_P(u) \quad \mbox{s.t.}\quad \frac{1}{2 }\norm{u}^2 -\displaystyle \frac{1}{2} =0,
\]
where $\ind_P$ denotes the indicator function of $P$. Let $\bu$ be a stationary point to this problem; equivalently, a stationary point to \eqref{minsupport}. Since the above problem satisfies the  linear independence constraint qualification, 
there exists $\lambda\in\R$ such that the following (generalized KKT) system is satisfied:
\begin{equation}\label{kkt}
\left\{
\begin{array}{llll}
0 & \in \partial (F_Q + \ind_P)(\bar u) - \lambda \bar u\\
\bar u & \in P,\;\; \norm{\bar u} =1.
\end{array}
\right.
\end{equation}
Observe that 
$\partial (  F_Q +\ind_P)( \bu) = \partial  F_Q(\bu) +\partial \ind_P( \bu) = \partial  F_Q(\bu) +N_P( \bu),$
due to the fact that $F_Q$ is convex and real-valued and $P\neq \emptyset$ is a closed convex set; see for instance  \cite[Corollaries 2.6 and 2.7]{WWBook}. Recall that $N_P(\bu)$ denotes the normal cone to $P$ at $\bu$, and since $P$ is a closed convex cone, $N_P( \bar u)=\{w \in P^\circ: w\perp \bar u\}$.
The first inclusion in~\eqref{kkt} ensures the existence of $\bar v \in \partial F_Q(\bar u)$ and $w \in N_P(\bar u)$ such that 
$0=\bar v + w -\lambda \bar u$. By multiplying this inequality by $\bar u$ and recalling that $\inner{w}{\bar u}=0$ for all $w \in N_P(\bar u)$, we get that $\lambda = \inner{\bar u}{\bar v}$. Thus, $-w=\bar v - \inner{\bar u}{\bar v} \bar u \in -N_P(\bar u)\subset - P^\circ = P^*$, and we can rewrite \eqref{kkt} as
\[
\left\{
\begin{array}{llll}
P&\ni \bar u \perp  \bar v  - \inner{\bar u}{\bar v} \bar u \in P^*\\
\partial F_Q(\bar u) &\ni  \bar v  ,\;\; \norm{\bar u} =1.
\end{array}
\right.
\]
Proposition~\ref{prop FQ} item (c) concludes \eqref{optim}. 

Furthermore, if $\arg\max_{v \in Q\cap S_n} \inner{\bar u}{v}$ is a singleton, then $\bar v$ in the stated system is the only  solution to $\max_{v \in Q\cap S_n} \inner{\bar u}{v}$. Thus, $\bar v$ must satisfy the following KKT system
\[
\left\{
\begin{array}{llll}
0 & \in -\bar u + N_Q(\bar v) - \gamma \bar v\\
\bar v & \in Q,\;\; \norm{\bar v} =1.
\end{array}
\right.
\]
By proceeding as before, we get $\gamma=-\inner{\bar u}{\bar v}$ and thus $-\bar u + \inner{\bar u}{\bar v}\bar v \in -N_Q(\bar v)\subset -Q^\circ$, i.e., $\bar u-\inner{\bar u}{\bar v}\bar v \in Q^\circ$. By replacing the inclusion $\bar v \in \partial F_Q(\bar u)$ above with these latter identities we get the stated system.
Proposition~\ref{prop FQ} items (f), (d) and (c) ensure that if $\bar u\not \in Q^\circ$, then $\arg\max_{v \in Q\cap S_n} \inner{\bar u}{v}$ is a singleton.
\end{proof}

\begin{remark}\label{rembd}
In general, a solution of \eqref{minsupport} may not lie on the boundary of $P$. Indeed, let $Q:=\{v\in\R^2: |v_1|\leq v_2\}$ (the Lorentz cone in $\R^2$) and let $P=-Q$. It is not difficult to see that $\mathfrak{s}(P,Q)=\cos(3\pi/4)$, and that $\bu=(0,-1)$ is a solution of \eqref{minsupport} since  $\proj_{Q\cap S_2}(\bu)=\{(1/\sqrt{2},1/\sqrt{2}),\;(-1/\sqrt{2},1/\sqrt{2})\}$, see \eqref{proj:rev}. Observe that $\bu$ is not in the boundary of $P$. Furthermore, the vector $\bv$ associated with $\bu$ in the optimality conditions described in Proposition\,\ref{prop:stat} may fail to have unit norm. Indeed, for the above example, $\bu=(0,-1)$ satisfies the optimality conditions \eqref{optim} with $\bv=(0,1/\sqrt{2})\in \mathtt{conv\,}\proj_{Q\cap S_2}(\bu)$. Observe that $\Vert \bv\Vert=1/\sqrt{2}.$
\end{remark}

It is clear that $-1\leq\mathfrak{s}(P,Q)\leq 1$ for every $P,Q\in\mathcal C_n$. Below, we describe when the extremes are attained. For $w\in\R^n$, we denote $\R_+ (w):=\{\alpha w:\alpha\geq 0\}$.
\begin{proposition}\label{prop:extreme}
Let $P,Q\in\mathcal C_n$. Then,
\begin{enumerate}
\item[(a)] $\mathfrak{s}(P,Q)=1$ if and only if $P\subseteq Q$.
\item[(b)] $\mathfrak{s}(P,Q)=-1$ if and only if $Q=\R_+(-\bu)$ for some $\bu\in P\cap S_n$.
\end{enumerate}
\end{proposition}
\begin{proof}
Item (a). Suppose that $\mathfrak{s}(P,Q)=1$. Then, $F_Q(u)=1$ for all $u\in P\cap S_n$. Hence, for all $u\in P\cap S_n$, there exists $v\in Q\cap S_n$ such that $\langle u,v\rangle =1$. This last means that $u=v$, so $u\in Q$. Therefore, $P\subseteq Q$. Conversely, if $P\subseteq Q$, then for $u\in P\cap S_n$ we can take $v=u\in Q\cap S_n$ which implies that $\langle u,v\rangle=1$. Thus, $F_Q(u)=1$ for all $u\in P\cap S_n$ which means $\mathfrak{s}(P,Q)=1$.\\
Item (b). Suppose that $\mathfrak{s}(P,Q)=-1$. Then, there exists $\bu\in P\cap S_n$ such that $F_Q(\bu)=-1$. Thus, from the definition of $F_Q$ and the Cauchy-Schwarz inequality we get $-1\leq \langle \bu,v\rangle\leq -1$ for all $v\in Q\cap S_n$. It means that $v=-\bu$ for all $v\in Q\cap S_n$. Therefore, $Q=\R_+(-\bu)$. Conversely, suppose that $Q=\R_+(-\bu)$ for some $\bu\in P\cap S_n$. Then, $F_Q(u)=-\langle u,\bu\rangle$ for all $u\in P\cap S_n$. Thus, $\mathfrak{s}(P,Q)=-1$.
\end{proof}

Next, we state sufficient conditions to ensure that any stationary point of \eqref{minsupport} is on the boundary of $P$ relative to some subspace.
\begin{proposition}\label{boundary}
Let $P,Q\in\mathcal C_n$ be such that $|\mathfrak{s}(P,Q)|\neq 1$. Let $L\subseteq\R^n$ be the smallest linear subspace containing both $P$ and $Q$. Let $\bu$ be a stationary point of the problem \eqref{minsupport} such that $\proj_{Q\cap S_n}(\bu)$ is a singleton. Then, $\bu$ is on the boundary of $P$ relative to $L$.
\end{proposition}
\begin{proof}
Reasoning by contradiction, suppose that $\bu$ is in the interior of $P$ relative to $L$. Then, there exists $\epsilon>0$ such that $\bu+\epsilon S_n\cap L\subset P$. From the necessary optimality conditions for $\bu$ given in Proposition\,\ref{prop:stat}, there is $\bv\in Q$ such that $(\bu,\bv)$ satisfies the system \eqref{optim}. Moreover, since we assume that $\proj_{Q\cap S_n}(\bu)$ ($=\arg\max_{v \in Q\cap S_n} \inner{\bar u}{v}$) is a singleton, $(\bu,\bv)$ satisfies \eqref{optim2}, and we deduce that $\mathfrak{s}(P,Q)=\langle\bu,\bv\rangle$ and $\Vert \bv\Vert=1$. We claim that $\bv-\langle \bu,\bv\rangle \bu\neq 0$. In fact, if $\bv-\langle \bu,\bv\rangle \bu= 0$ then,
\[1=\Vert \bv\Vert=|\langle \bu,\bv\rangle|\Vert \bu\Vert=|\langle \bu,\bv\rangle|,\] 
which is not possible because of  $|\mathfrak{s}(P,Q)|\neq 1$.   Let $w:=(\langle \bu,\bv\rangle \bu-\bv)/\Vert \langle \bu,\bv\rangle \bu-\bv\Vert\in S_n\cap L$. Then, $\bu+\epsilon w\in P.$ From \eqref{optim2} we have
\[0\leq\langle \bu+\epsilon w,\bv-\langle \bu,\bv\rangle \bu\rangle=\epsilon\langle w,\bv-\langle \bu,\bv\rangle \bu\rangle=-\epsilon\Vert \bv-\langle \bu,\bv\rangle \bu\Vert\leq0,\]
obtaining $\bv-\langle \bu,\bv\rangle \bu=0$ which is a contradiction.
\end{proof}

\section{Special cases}\label{sec:cases}

\subsection{Dissimilarity between linear subspaces}
Let $P$ and $Q$ be linear subspaces in $\RR^n$ of dimensions $p$ and $q$, respectively, given by
\begin{equation}\label{subspaces}
P=U(\RR^p)\quad\text{and}\quad Q=V(\R^q),
\end{equation}
where $U\in\R^{n\times p}$ and $V\in\R^{n\times q}$ are matrices with their columns forming orthonormal basis for $P$ and $Q$, respectively.

 The principal angles  $\theta_1,\ldots,\theta_m$ between $P$ and $Q$ are defined recursively by
 \begin{eqnarray}\label{thetak}
 \cos \theta_k\, =  \max_{ u\in P_k \cap S_n\,,\,  v\in Q_k\cap  S_n}\; \langle u,v\rangle,
 \end{eqnarray}
 where $m:= \min\{ p,q\}$ and
 \begin{equation*}
\left\{\begin{array}{lll}
 P_1:=P, \;Q_1:=Q, \\ [1,2mm]
P_{k+1}:= \{w\in P_k: \langle u_k,w\rangle =0\},\\ [1,2mm]
Q_{k+1}:= \{w\in Q_k: \langle v_k,w\rangle =0\},\\ [1,2mm]
 (u_k,v_k) \mbox{ solution to } (\ref{thetak}).
\end{array}
\right.
\end{equation*}
Observe that $0\leq\theta_1\leq\theta_2\leq\cdots\leq\theta_m\leq\pi/2$. The largest principal angle between $P$ and $Q$ will be denoted by $\theta_{\max}(P,Q)$. When $P$ and $Q$ are of the form \eqref{subspaces}, from \cite[Theorem\,1]{Bjorck} (see also \cite[Theorem\,6.2]{SeSo1}) we have that $\theta_k$ is a principal angle of $(P,Q)$ if and only if $\cos\theta_k$ is a singular value of $V^\top U$. Furthermore, $\cos(\theta_m)\leq\cos(\theta_{m-1})\leq\cdots\leq\cos(\theta_1)$. Thus, $\cos(\theta_{\max}(P,Q))$ is the smallest singular value of $V^\top U$.

Computing the separation between linear subspaces is a classic problem in linear algebra. For this purpose, many distances were constructed, and most of them are based on the principal angles between subspaces, see \cite{YeLim}. In particular, 
for a pair of subspaces of the same dimension $P$ and $Q$, the spectral distance between them is defined as
\begin{equation}\label{spec}
d_{\rm spec}(P,Q)=2\sin[\theta_{\max}(P,Q)/2].
\end{equation}
In the next theorem, we describe the distance $\Dis_r(P,Q)$ when $P$ and $Q$ are linear subspaces.
\begin{theorem}
Let $1\leq r\leq\infty$, and let $P=U(\R^p)$  and $Q=V(\R^q)$ be linear subspaces of $\R^n$ as in \eqref{subspaces}. Then,
\begin{align*}
&\displaystyle {\rm Dis}_r(P,Q)= 2\left\Vert\left(\frac{\sqrt{2}}{2}, \sin\left[\frac{\theta_{\max}(P,Q)}{2}\right]\right)\right\Vert_r,\quad\text{if }p \neq q,\\
&\displaystyle {\rm Dis}_r(P,Q)= 2\Vert(1,1)\Vert_r\sin\left[\frac{\theta_{\max}(P,Q)}{2}\right],\quad\text{if } p=q,
\end{align*}
where $\theta_{\max}(P,Q)\in[0,\pi/2]$ is the largest principal angle between $P$ and $Q$; that is, $\cos(\theta_{\max}(P,Q))$ is the smallest singular value of $V^\top U$.

\end{theorem}
\begin{proof}
Let $u\in\R^n$ be such that $\Vert u\Vert=1$. Let us compute the projection of $u$ onto $Q\cap S_n$:
\begin{eqnarray*}
\bv\in \proj_{Q\cap S_n}(u)&\Leftrightarrow&\bv\in\argmax_{v\in Q\cap S_n}\,\langle u,v\rangle\,\Leftrightarrow\, \bv=V\bar y\;\text{ with }\; \by\in\argmax_{y\in S_q}\,\langle y,V^\top u\rangle,\\
&\Leftrightarrow&\bv=\frac{VV^\top u}{\Vert V^\top u\Vert}\,\text{ if }\,V^\top u\neq 0; \,\text{otherwise}\;\bv\in Q\cap S_n.
\end{eqnarray*}
From Proposition\,\ref{prop FQ}(e) we know that $F_Q(u)=\langle u,\bv\rangle$ with $\bv\in \proj_{Q\cap S_n}(u)$. Therefore, from the above computation, we deduce that
\[F_Q(u)=\begin{cases}
\Vert V^\top u\Vert,&\text{if }u\notin Q^\perp\\
0,&\text{if }u\in Q^\perp
\end{cases}.\]
Hence, if $P\cap Q^\perp=\{0\}$,
\[\mathfrak{s}(P,Q)=\min_{u\in P\cap S_n}F_Q(u)=\min_{u\in P\cap S_n}\Vert V^\top u\Vert=\min_{x\in S_p}\Vert V^\top Ux\Vert=\sigma_{\min}(V^\top U),\]
where $\sigma_{\min}(V^\top U)$ denotes the smallest singular value of $V^\top U$. If $P\cap Q^\perp\neq\{0\}$, it is clear that $\mathfrak{s}(P,Q)=0$. Thus,

\[\mathfrak{s}(P,Q)=\begin{cases}
\sigma_{\min}(V^\top U),&\text{if }P\cap Q^\perp=\{0\},\\
0,&\text{if }P\cap Q^\perp\neq\{0\}.
\end{cases}\]
Recall that $\theta_{\max}(P,Q)=\arccos(\sigma_{\min}(V^\top U))$. Therefore, we deduce that
\begin{equation}\label{theta-arccos}
\Theta(P,Q)=\arccos(\mathfrak{s}(P,Q))=\begin{cases}
\theta_{\max}(P,Q),&\text{if }P\cap Q^\perp=\{0\},\\
\frac{\pi}{2},&\text{if }P\cap Q^\perp\neq\{0\}.
\end{cases}
\end{equation}
By interchanging the roles of $P$ and $Q$, and recalling that $\sigma_{\min}(V^\top U)=\sigma_{\min}(U^\top V)$, we also get that
\begin{equation}\label{theta-arccos2}
\Theta(Q,P)=\begin{cases}
\theta_{\max}(P,Q),&\text{if }P^\perp\cap Q=\{0\},\\
\frac{\pi}{2},&\text{if }P^\perp\cap Q\neq\{0\}.
\end{cases}
\end{equation}

Recall that $p=\mathtt{dim} P$ and $q=\mathtt{dim} Q$, and set $E:=V^\top U\in\R^{q\times p}$. Suppose that $p>q $. The Rank-Nullity Theorem says that $\mathtt{dim}({\rm Ker}(E))+\mathtt{rank}(E)=p$. Since $\mathtt{rank}(E)\leq q<p$, we deduce that $\mathtt{dim}({\rm Ker}(E))\geq 1$. It implies that there exists a nonzero $x\in\R^p$ such that $V^\top U x=0$. Thus, $Ux\in P\cap Q^\perp$, and from \eqref{theta-arccos} we deduce that $\Theta(P,Q)=\pi/2$. Now, to compute $\Theta(Q,P)$, suppose that $P^\perp\cap Q\neq\{0\}$. Then, there exists a nonzero $y\in\R^q$ such that $U^\top Vy=0$. Hence, $\mathtt{dim}({\rm Ker}(E^\top))\geq 1$. The Rank-Nullity Theorem for $E^\top$ says that $\mathtt{dim}({\rm Ker}(E^\top))+\mathtt{rank}(E)=q$. Then, $\mathtt{rank}(E)<q$ which means that $E$ is not of complete rank. Hence, $\sigma_{\min}(E)=0$, and then $\Theta(Q,P)=\theta_{\max}(P,Q)=\arccos(0)=\pi/2$. Therefore, by combining it with \eqref{theta-arccos2} we get that $\Theta(Q,P)=\theta_{\max}(P,Q)$ when either $P^\perp\cap Q\neq\{0\}$ or $P^\perp\cap Q=\{0\}$. Thus, 
\begin{equation}\label{dis}
\displaystyle {\rm Dis}_r(P,Q)= 2\left\Vert\left(\frac{\sqrt{2}}{2}, \sin\left[\frac{\theta_{\max}(P,Q)}{2}\right]\right)\right\Vert_r.
\end{equation}
By symmetry, the above result is also valid when $p<q$. Hence, \eqref{dis} holds whenever $p\neq q$.

Suppose that $p=q$. Then, analogous to the above reasoning, the matrix $E=V^\top U\in\R^{p\times p}$ is not of complete rank when $P\cap Q^\perp\neq\{0\}$. Thus, $\sigma_{\min}(E)=0$, and then $\Theta(P,Q)=\theta_{\max}(P,Q)=\arccos(0)=\pi/2$. By combining it with \eqref{theta-arccos} we get that $\Theta(P,Q)=\theta_{\max}(P,Q)$ in the both cases of \eqref{theta-arccos}. Analogously, we also have $\Theta(Q,P)=\theta_{\max}(P,Q)$. Hence, if $p=q$ then
\begin{eqnarray*}
\displaystyle {\rm Dis}_r(P,Q)&=& 2\left\Vert\left(\sin\left[\frac{\theta_{\max}(P,Q)}{2}\right], \sin\left[\frac{\theta_{\max}(P,Q)}{2}\right]\right)\right\Vert_r\\
&=&2\Vert(1,1)\Vert_r\sin\left[\frac{\theta_{\max}(P,Q)}{2}\right].
\end{eqnarray*}
The proof is complete.
\end{proof}

\begin{remark}
We observe that $\Dis_{\infty}(P,Q)=d_{\rm spec}(P,Q)$ when $P$ and $Q$ are linear subspaces of equal dimensions. That is, the above theorem reveals that the spectral distance \eqref{spec} is induced by the Pompeiu-Hausdorff distance.
\end{remark}

\subsection{Dissimilarity between revolution cones}\label{Sec:revolution}
A revolution cone in $\R^n$ is a closed convex cone defined as
\[{\rm Rev}(\phi,b):=\{u\in\R^n:\,\langle b,u\rangle\geq\Vert u\Vert\cos\phi\},\]
where $b\in S_n$ defines the revolution axis, and $\phi\in[0,\pi/2)$ is the half-aperture angle.
\begin{theorem}\label{th:rev}
Let $P={\rm Rev}(\phi_1,b_1)$ and $Q={\rm Rev}(\phi_2,b_2)$, with $b_1,b_2\in S_n$, and $\phi_1,\phi_2\in[0,\pi/2)$. Then,
\[\Theta(P,Q)=\max\left\{0,\,\min\{\pi, \arccos\langle b_1,b_2\rangle +\phi_1\}-\phi_2\right\}.\]
\end{theorem}
\begin{proof}
Without loss of generality, we may assume that $b_2=e_n$, where $e_n$ denotes the $n$th canonical vector in $\R^n$. Indeed, we can take an orthogonal transformation $V\in\R^{n\times n}$ so that $Vb_2=e_n$, $V(P)$ and $V(Q)$ remain being revolution cones of half-angle aperture $\phi_1$ and $\phi_2$, respectively, and $\Theta(V(P),V(Q))=\Theta(P,Q)$ because of Proposition\,\ref{pr:elementary}(b). Thus, $Q$ can be written as 
\[Q=\{(\tv,v_n)\in\R^{n-1}\times\R:\Vert\tv\Vert\leq v_n\tan\phi_2\},\]
where $\tilde v:=(v_1,\ldots,v_{n-1})$ for $v=(v_1,\ldots,v_n)\in\R^n$. Indeed,
\begin{eqnarray*}
v\in Q  &\Leftrightarrow&  \langle e_n,v\rangle\geq\Vert v\Vert\cos\phi_2 \; \Leftrightarrow \;v_n\geq 0,\, v_n^2\geq (\Vert \tilde v\Vert^2+v_n^2)\cos^2\phi_2 \\
&\Leftrightarrow  &v_n\geq 0,\,(\tan^2\phi_2)v_n^2\geq \Vert \tilde v\Vert^2\; \Leftrightarrow \;\Vert\tv\Vert\leq v_n\tan\phi_2.
\end{eqnarray*}
From \cite[Example\,9.1]{Bauschke}, for $u=(\tilde u,u_n)\in P\cap S_n$, we have
\begin{equation}\label{proj:rev}
\proj_{Q\cap S_n}(u)=\begin{cases}
u,&\text{if }u\in Q,\\
\left\{(\tilde w,\cos\phi_2):\Vert\tilde w\Vert=\sin\phi_2\right\},&\text{if }u=(0,-1),\\
\left(\frac{\sin\phi_2}{\Vert\tu\Vert}\tu,\cos\phi_2\right),&\text{otherwise}.
\end{cases}
\end{equation}
From Proposition\,\ref{prop FQ}(e) we have that $F_Q(u)=\langle u,v'\rangle$ for any $v'\in\proj_{Q\cap S_n}(u)$. Hence,
\begin{equation}\label{FQrev}
F_Q(u)=\begin{cases}
1,&\text{if }u\in Q,\\
-\cos(\phi_2),&\text{if }u=(0,-1),\\
\Vert\tu\Vert\sin\phi_2 + u_n\cos\phi_2,&\text{otherwise}.
\end{cases}
\end{equation}
Let $\beta:=\arccos\langle b_1,b_2\rangle\in[0,\pi].$ We proceed to solve the problem \eqref{minsupport}, which becomes: 
\begin{equation}\label{prob:rev}
\mathfrak{s}(P,Q)=\min_u F_Q(u)\quad
\text{s.t.}\quad\langle b_1, u\rangle\geq\cos \phi_1,\;\Vert u\Vert=1.
\end{equation}
We consider three cases according to the values of $\phi_1$, $\phi_2$ and $\beta$:\\
\emph{Case 1.} Suppose that $\beta+\phi_1\leq\phi_2$. It means that $P\subseteq Q$. Then, from Proposition\,\ref{prop:extreme} we deduce $\mathfrak{s}(P,Q)=1$.\\
\emph{Case 2.} Suppose that $\pi\leq\beta+\phi_1$. Then $\bu:=(0,-1)\in P$, and from \eqref{FQrev} we have $F_Q(\bu)= -\cos(\phi_2)$. Furthermore, notice that $-\cos(\phi_2)\leq F_Q(u)$ for all $u\in P\cap S_n$. Indeed, it comes directly from \eqref{FQrev} when $u\in Q$ or $u=(0,-1)$. For the remaining case, it comes from
\[F_Q(u)+\cos\phi_2=\Vert\tu\Vert\sin\phi_2 + (u_n+1)\cos\phi_2\geq 0,\]
where the inequality is because $-1\leq u_n\leq 1$ since $\Vert u\Vert=1$. Therefore, $\mathfrak{s}(P,Q)=-\cos(\phi_2)$.\\
\emph{Case 3.} Suppose that $\phi_2<\beta+\phi_1<\pi$. Then, $P\not\subseteq Q$ and $(0,-1)\notin P$, which imply $|\mathfrak{s}(P,Q)|\neq 1$, see Proposition\,\ref{prop:extreme}. Furthermore, because of \eqref{proj:rev}, those facts also imply that $\proj_{Q\cap S_n}(u)$ is a singleton for any optimal point $u$ of \eqref{prob:rev}. Now, let $L$ be the smallest linear subspace containing $P$ and $Q$. Observe that $L=\mathtt{span}\{b_1,b_2\}$ when $\phi_1=\phi_2=0$, and $L=\R^n$ otherwise. Hence, from Proposition\,\ref{boundary} we have that any solution to problem \eqref{prob:rev} is in the boundary of $P$ relative to $L$. Thus, this problem can be written as
\begin{equation}\label{boundary0}
\min_{u} \langle u, v(u)\rangle,\quad\text{s.t. }\langle b_1, u\rangle=\cos \phi_1,\;\Vert u\Vert=1,
\end{equation}
where $v(u):=\left(\frac{\sin\phi_2}{\Vert\tu\Vert}\tu,\cos\phi_2\right)$ is the projection of $u$ onto $Q\cap S_n$, see \eqref{proj:rev} and \eqref{FQrev}.  Suppose that $b_1=b_2$. In this case, $u=(\tu,u_n)\in S_n$ is on the boundary of $P$ if and only if $\Vert\tu\Vert=\sin\phi_1$ and $u_n=\cos\phi_1$. Therefore, \[F_Q(u)=\Vert\tu\Vert\sin\phi_2 + u_n\cos\phi_2=\sin\phi_1\sin\phi_2+\cos\phi_1\cos\phi_2=\cos(\phi_1-\phi_2).\] Thus, $\mathfrak{s}(P,Q)=\cos(\phi_1-\phi_2)$. Suppose that $b_1\neq b_2$, then $\{b_1,b_2\}$ is linear independent since $b_1\neq-b_2$ because of $-b_2=(0,-1)\notin P$. The KKT optimality conditions for $u$ to be solution of \eqref{boundary0} say that there exits $\lambda_1,\lambda_2\in\R$ such that 
\begin{equation}\label{linear1}
v(u)=\lambda_1 b_1+\lambda_2 u.
\end{equation}
On the other hand, since the optimal $u$ is not in $Q$, otherwise $\mathfrak{s}(P,Q)=1$, $v(u)$ is on the boundary of $Q$. Then, $v(u)$ solves $\max_v\langle u,v\rangle$ s.t. $\langle b_2, v\rangle=\cos \phi_2,\;\Vert v\Vert=1$. Hence, the KKT optimality conditions for $v(u)$ say that there exist $\mu_1,\mu_2\in\R$ such that 
\begin{equation}\label{linear2}
u=\mu_1 b_2+\mu_2 v(u).
\end{equation}
By combining \eqref{linear1} and \eqref{linear2} we deduce that $u,v(u)\in\mathtt{span}\{b_1,b_2\}$. Therefore, the problem \eqref{boundary0} becomes a 2-dimensional problem. Let $\tilde b_1:=(\csc\beta) b_1-(\cot\beta)b_2$, then $\{\tilde b_1,b_2\}$ is an orthonormal basis of $\mathtt{span}\{b_1,b_2\}$. There are only two unit vectors on the boundary of $P$ that are in $\mathtt{span}\{b_1,b_2\}$: $\hat u:=(\sin(\beta-\phi_1))\tilde b_1+(\cos(\beta-\phi_1))b_2$ which gives $F_Q(\hat u)=\cos(\beta-\phi_1-\phi_2)$, and $\bu:=(\sin(\beta+\phi_1))\tilde b_1+(\cos(\beta+\phi_1))b_2$ which gives $F_Q(\bu)=\cos(\beta+\phi_1-\phi_2)$. Since $F_Q(\bu)\leq F_Q(\hat u)$, then  $\mathfrak{s}(P,Q)=\cos(\beta+\phi_1-\phi_2)$.
Summarizing, we have obtained:
\[\mathfrak{s}(P,Q)=
\begin{cases}
\cos(0),&\text{if }\beta+\phi_1\leq\phi_2,\\
\cos(\pi-\phi_2),&\text{if }\pi\leq\beta+\phi_1,\\
\cos(\beta+\phi_1-\phi_2),&\text{if }\phi_2<\beta+\phi_1<\pi.
\end{cases}
\]
From this, we obtain the announced formula for $\Theta(P,Q)$.
\end{proof}

\section{Computing the dissimilarity measure: the polyhedral setting}\label{sec:poly}

In this section, $P,\,Q\in\mathcal C_n$ are assumed to be polyhedral cones. That is, they are convex, closed, and finitely generated cones. Hence, without loss of generality, we assume that
\begin{equation}\label{polyhedral}
P=G (\mathbb R^p_+) = \{ Gx \ | \ x \geq 0 \} 
\quad \mbox{and} \quad 
Q=H (\mathbb R^q_+)  = \{ Hy \ | \ y \geq 0 \}, 
\end{equation}
where $G=[g_1,\ldots,g_p]\in\R^{n\times p}$ and $H=[h_1,\ldots,h_q]\in\R^{n\times q}$ are matrices whose columns 
are unitary, that is, $\|g_i\| = 1$ for all $i$ and $\|h_j\| = 1$ for all $j$. 

To compute $F_Q(u)$, we need to describe $\proj_{Q\cap S_n}(u)$. This is given in the next proposition, which is borrowed from \cite[Corollary\,8.6]{Bauschke}. For a vector $w\in\R^n$, $ \max (w)$ denotes the largest entry of $w$, and for vectors vectors $w_1,\ldots,w_k\in\R^n$, $\mathtt{cone}\{w_1,\ldots,w_k\}:=\{\alpha_1 w_1+\cdots+\alpha_k w_k:\alpha_i\geq 0\}.$
\begin{proposition}({\cite[Corollary\,8.6]{Bauschke}})\label{prop:FQ2}
Let $Q$ be a polyhedral cone as in \eqref{polyhedral}, and let $u\in S_n$. Let $J(u):=\mathop{\rm arg\,max}\limits_{j=1,\ldots,q} \,\langle u,h_j\rangle$. Then, $\partial F_Q(u) = \mathtt{conv}\, \proj_{Q\cap S_n}(u)$, where,
\[
\proj_{Q\cap S_n}(u)=
\left\{
\begin{array}{ll}
     \displaystyle \frac{\proj_{Q}(u)}{\norm{\proj_{Q}(u)}}& \mbox{ if } \max( H^\top u) >0,  \\
     \mathtt{cone}\{h_j:j\in J(u)\}\cap S_n& \mbox{ if } \max( H^\top u) =0,  \\
     \{h_j:j\in J(u)\} & \mbox{ if } \max( H^\top u)< 0 .
\end{array}
\right.
\]
\end{proposition}

This section deploys an optimization approach to tackle problem~\eqref{maxminPQ} by dealing with its equivalent problem of minimizing the support function given in \eqref{minsupport}; that is, $\min_{u\in P\cap S_n}F_Q(u)$ (recall that this problem has a solution because $F_Q$ is continuous --convex and real-valued-- and $P\cap S_n$ is compact).
Although the problem of defining $F_Q$ is nonconvex, Propositions~\ref{prop FQ}(e) and~\ref{prop:FQ2} show that, in the polyhedral setting—and in the most complex scenario\footnote{If $\max( H^\top u)\leq 0$ any generator $h_{i^*}$ with $i^*\in J(u)=\arg\max_{j=1,\ldots,q} \, \langle u,h_j\rangle$ is a global solution.}—computing a global solution reduces to solving a strictly convex quadratic problem yielding $\proj_Q(u)$ (and then normalizing). When solving such a subproblem, a subgradient of $F_Q$ at $u$ is readily available in Proposition\,\ref{prop:FQ2}.

Despite this useful property, solving problem~\eqref{minsupport} (and thus \eqref{maxminPQ}) remains a challenging nonsmooth, nonconvex optimization problem. In what follows, we investigate an approach to tackle this problem: a global optimization methodology that can be conveniently converted to a local-solution strategy.

\subsection{Cutting-plane approach}
Given that function $F_Q$ in~\eqref{support} is convex but nonsmooth, we may solve the problem by a cutting-plane method that, at iteration $k$, replaces $F_Q(u)$ with the cutting-plane model:
\[
\check F_Q^k(u)=\max_{j=1,\ldots,k} \{F_Q(u^j)+\inner{v^j}{u-u^j}\}\leq F_Q(u).
\]
In this notation,  $ v^j \in \proj_{Q\cap S_n}(u^j)\subseteq\partial F_Q(u^j)$ is as in Proposition~\ref{prop:FQ2}. Thus, $F_Q(u^j)=\inner{ v^j}{u^j}$, see Proposition\,\ref{prop FQ}(e), and the cutting-model boils down to 
\[
\check F_Q^k(u)=\max_{j=1,\ldots,k} \inner{ v^j}{u}.
\]
Accordingly, we get the following master program
\[
    u^{k+1} \in \arg\min_{u \in P} \; \check F_Q^k(u) \quad \mbox{s.t.}\quad \norm{u}=1,
\]
which is equivalent to the following quadratic-constrained problem
\begin{equation}\label{master}
    \left\{
    \begin{array}{llll}
    \displaystyle \min_{(u, r) \in \R^{n+1}} & r \\
     \mbox{s.t.}&  u\in P, \; \norm{u}^2=1\\
                & \inner{v^j}{u}\leq r,\; j=1,\ldots,k.
    \end{array}
    \right.
\end{equation}
Commercial solvers such as Gurobi can globally and efficiently solve problem~\eqref{master} in low dimensions, typically up to a few dozen variables.

Let $(u^{k+1},r^{k+1})$ be a solution to this problem. Since the cutting-plane model is a lower approximation of the convex function $F_Q$, we get that $r^{k+1}$ is a lower bound on $\cos[\Theta(P,Q)]$, the optimal value of~\eqref{minsupport}. As any feasible point gives an upper bound, the difference between these bounds gives a way to terminate the cutting-plane approach presented in Algorithm~\ref{CP}.

\begin{algorithm}
\caption{Cutting-plane algorithm for computing the angle $\Theta(P,Q)$}
\label{CP}
\begin{algorithmic}[1]
\STATE  Given $\mathtt{tol}> 0$ and $u^1 \in P\cap S_n$, compute $v^1 \in  \proj_{Q\cap S_n}(u^1)$ 
\FOR{$k=1,2,3,\ldots$}  
\STATE Let $(u^{k+1},r^{k+1})$ be a solution to the master program~\eqref{master} \label{lin:master}
\STATE Compute  $v^{k+1} \in \proj_{Q\cap S_n}(u^{k+1})$
\IF{$\inner{u^{k+1}}{v^{k+1}}-r^{k+1}\leq \mathtt{tol}$}
\STATE Stop: return the angle $\arccos{\inner{ u^{k+1}}{v^{k+1}}}$
\ENDIF
\ENDFOR
  \end{algorithmic}
\end{algorithm}

Note that computing $v^{k+1}$ is a relatively simple task: according to Proposition~\ref{prop:FQ2}, in the worst case, it amounts to projecting $u^{k+1}$ onto $Q$ (a convex programming problem) followed by normalization.
Algorithm~\ref{CP} is a specialized version of the celebrated Kelley's cutting-plane \cite{Kelley_1960}, which is known to compute a global solution provided the objective function is a real-valued convex function and the feasible set is compact (not necessarily convex).
The interested reader is referred to \cite[Ch.~10]{WWBook} for a more general cutting-plane method and its convergence analysis.
\begin{theorem}\label{theo:CP}
    Consider Algorithm~\ref{CP} with $\mathtt{tol} =0$. Then either the algorithm stops at iteration $k$ with a global solution $u^{k+1}$ to~\eqref{minsupport},  $(u^{k+1},v^{k+1})$  solving~\eqref{maxminPQ} and $r^{k+1}=\inner{u^{k+1}}{v^{k+1}}=\cos[\Theta(P,Q)]$, or the algorithm continues indefinitely. In the latter case, every cluster point $\bar u$ of the sequence $\{ u^{k}\}$ produced by the algorithm is a global solution to~\eqref{minsupport},  $(\bar u, \bar v)$ with $\bar v \in \proj_{Q\cap S_n}(\bar u)$ solves~\eqref{maxminPQ} and $\lim_{k} r^k =\cos[\Theta(P,Q)]$.
\end{theorem}
\begin{proof}
If the algorithm stops at $k$, then $\cos[\Theta(P,Q)]\leq F_Q(u^{k+1})= \inner{u^{k+1}}{v^{k+1}} $ $= r^{k+1}= \min_{u \in P\cap S_n} \check F^k_Q(u) \leq \min_{u \in P\cap S_n} F_Q(u)=\cos[\Theta(P,Q)] $. Therefore, $u^{k+1}$ solves \eqref{minsupport} and  $(u^{k+1},v^{k+1})$  solves~\eqref{maxminPQ} due to the equivalence between~\eqref{minsupport} and \eqref{maxminPQ}. Suppose now that the algorithm loops indefinitely. Let $\bar u$ be an arbitrary cluster point of the (bounded) sequence $\{u^k\}$, and let $\mathcal{K}$  be an index set such that $ \bar u= \lim_{\mathcal{K}\ni k \to \infty} u^k$.
    The fact that $\bar u$ solves~\eqref{minsupport} follows from the standard analysis of the cutting-plane method; see for instance \cite[Th.~10.1]{WWBook}.
     Indeed, one can show that the optimality gap $F_Q(u^{k})-\check F^{k-1}_Q(u^{k})\geq F_Q(u^{k}) - \cos[\Theta(P,Q)]\geq 0$ vanishes along the iterations $k \in \mathcal{K}$ (see the proof of Theorem~\ref{theo:CP2} below for more insights). Hence, continuity of $F$ implies that $F(\bar u)=\lim_{\mathcal{K} \ni k\to \infty} F(u^{k})= \cos[\Theta(P,Q)]$, the optimal value of~\eqref{minsupport}.
     Let $\mathcal{K}$  be an index set such that $ \bar u= \lim_{\mathcal{K}\ni k \to \infty} u^k$. As the projection onto a nonempty closed set is an outer-semicontinuous operator \cite[Example 5.23]{Rock-Wets}, we get that  any cluster point $\bar v$ of $\{v^k\}_{\mathcal{K}}$ belongs to $ \proj_{Q\cap S_n}(\bar u)$. Hence, Proposition~\ref{prop FQ} items (c) and (e), and the equivalence of problems~\eqref{maxminPQ} and \eqref{minsupport}  conclude the proof.
\end{proof}

When the solver used in Algorithm~\ref{CP} (line~\ref{lin:master}) returns a stationary point instead of solving the master problem globally, Algorithm~\ref{CP} still asymptotically computes a pair of points that satisfy the necessary optimality conditions given in Proposition~\ref{prop:stat}, as demonstrated by the following result.
\begin{theorem}\label{theo:CP2}
    Consider Algorithm~\ref{CP} with $\mathtt{tol} =0$ and suppose that, for all iterations $k$, $u^{k+1}$ is a stationary point to
    \[
    \min_{u \in P} \; \check F_Q^k(u) \quad \mbox{s.t.}\quad \norm{u}=1,
    \]
    and $r^{k+1}= \check F_Q^k(u^{k+1})$. 
    Then either the algorithm stops at iteration $k$ with a pair  $(u^{k+1},v^{k+1})$ satisfying the necessary optimality conditions~\eqref{optim}, or the algorithm continues indefinitely. In the latter case,
    let $\bar u$ be an arbitrary cluster point of the sequence $\{u^{k}\}$ produced by the algorithm. Then $(\bar u, \bar v)$ with $\bar v \in \proj_{Q\cap S_n}(\bar u)$ satisfies~\eqref{optim}. 
\end{theorem}
\begin{proof}
    By proceeding as in the first part of the proof of Proposition~\ref{prop:stat} (with $F_Q$ replaced with $\check F_Q^k$), we conclude that $u^{k+1}$ satisfies 
    \[
\left\{
\begin{array}{llll}
P&\ni u^{k+1} \perp   v^{k+1}  - \inner{u^{k+1}}{ v^{k+1}} u^{k+1} \in P^*,\\
\partial \check F^k_Q(u^{k+1}) &\ni   v^{k+1}  ,\;\; \norm{u^{k+1}} =1.
\end{array}
\right.
\]
Since $\check F^k_Q \leq F_Q$ are convex functions, we have that
any $\check v \in \partial \check F_Q^k(u^{k+1})$ belongs to the approximate subdifferential $\partial_{e_k} F_Q(u^{k+1}) :=\{v:\, F_Q(u^{k+1}) +\inner{v}{u-u^{k+1}}-e_k \leq F_Q(u)\}$,  with $e_k:=  F_Q(u^{k+1}) - \check F_Q^k(u^{k+1}) \geq 0$. Hence, the above system implies
    \begin{equation}\label{aux0}
\left\{
\begin{array}{llll}
P&\ni u^{k+1} \perp   v^{k+1}  - \inner{u^{k+1}}{v^{k+1}} u^{k+1} \in P^*,\\
\partial_{e_k} F_Q(u^{k+1}) &\ni  v^{k+1}  ,\;\; \norm{u^{k+1}} =1.
\end{array}
\right.
\end{equation}
If the algorithm stops at iteration $k$, then $e_k=0$. As $v^{k+1} \in \partial F_Q(u^{k+1})$, Proposition~\ref{prop FQ} items (c) and (d) asserts that $v^{k+1}  \in \mathtt{conv}\, \arg\max_{v\in Q\cap S_n}\inner{u^{k+1}}{v}$.
Thus, $(u^{k+1},v^{k+1})$ satisfies~\eqref{optim}. Suppose now that the algorithm loops indefinitely.
Next we prove that $\lim_k e_k =0$. To this end, observe that for any arbitrary index $i \in \{1,\ldots,k-1\}$,
\begin{align*}
0 & \leq e_k=  F_Q(u^{k+1}) - \check F_Q^k(u^{k+1}) = [F_Q(u^{k+1})-\check F^k_Q(u^{i})] +[\check F^k_Q(u^{i}) - \check F_Q^k(u^{k+1})]\\
& =  [F_Q(u^{k+1})-F_Q(u^{i})] +[\check F^k_Q(u^{i}) - \check F_Q^k(u^{k+1})]\\ 
&\leq \norm{u^{k+1}-u^i}+ \norm{u^{k+1}-u^i},
\end{align*}
where the last  equality follows from the fact that $i<k+1$ (thus $\check F^k_Q(u^{i}) = F_Q(u^{i})$), and the inequality holds because both $F_Q$ and $\check F^k_Q$ are Lipschitz continuous with Lipschitz constant equal to one. This last result follows from Proposition~\ref{prop FQ}(b) and the fact that, being a cutting-plane function, $\partial \check F^k_Q(u)$ is contained in the convex hull of $\{v^1,\ldots, v^k\}$.
As a result, if there would exist $\epsilon >0$ such that $\epsilon \leq  e_k$ for all $k$ large enough, then the bounded sequence $\{u^k\}$ would not have a convergent subsequence, which is impossible in $\R^n$. 
Hence, $\lim_k e_k =0$. (Thus the algorithm terminates after finitely many iterations provided $\mathtt{tol}>0$.)

Let $\mathcal{K} \subset \mathbb{N}$ be such that $\lim_{\mathcal{K} \ni k \to \infty} u^{k+1} = \bar u$.
As the approximated subdifferential is locally bounded \cite[Lem. 2.10]{WWBook} and $\lim_k e_k =0$, there exists $\mathcal{K}' \subset \mathcal{K}$ and $\bar v $ such that $\lim_{\mathcal{K}'\ni k \to \infty} v^{k+1}=\bar v$.
By passing to the limit as $k \in \mathcal{K}'$ in the system~\eqref{aux0} and recalling that the approximate subdifferential, the normal cone, and the projection onto a nonempty closed set are outer-semicontinuous \cite[Prop 2.7, Lem. 3.2]{WWBook} and \cite[Example 5.23]{Rock-Wets}, we conclude that
\[
\left\{
\begin{array}{llll}
P&\ni \bar  u \perp  \bar  v  - \inner{\bar u}{\bar  v} \bar u \in P^*,\\
\partial F_Q(\bar u) &\ni  \bar v  ,\;\; \norm{\bar u} =1.
\end{array}
\right.
\]
As $\bar v \in \partial F_Q(\bar u)$, Proposition~\ref{prop FQ}(c)-(d)  asserts that $\bar v  \in \mathtt{conv}\, \arg\max_{v\in Q\cap S_n}\inner{\bar u}{v}$. Hence,  the pair $(\bar u,\bar v)$ satisfies the necessary optimality conditions in~\eqref{optim}.
\end{proof}

Note that the master program in the above theorem is a nonsmooth nonconvex optimization problem. To compute a stationary point $u^{k+1}$, we may consider the smooth reformulation~\eqref{master} and apply standard NLP solvers. 

\subsubsection*{Warm start}
Employing a global optimization solver for the master program \eqref{master} in Algorithm~\ref{CP} can lead to prohibitive computational times, even when the problem dimension $n$ is only a few dozen. To warm-start the algorithm (even in the local-solution setting described in Theorem~\ref{theo:CP2}) we adopt a simple yet effective strategy outlined in Algorithm~\ref{Heur}. This approach aims to provide a good-quality initial point along with an initial cutting-plane model. Specifically, we begin by examining the generators of the set $P$.\begin{algorithm}
\caption{Warm-start for the cutting-plane method}
\label{Heur}
\begin{algorithmic}[1]
\STATE  Let  $G \in \R^{n\times p}$ and $H\in \R^{n\times q}$ be the matrices of generators of cones $P$ and $Q$
\FOR{$j=1,2,\ldots, p$} 
\STATE Set $u^j\gets g_j$ and $v^j\in \proj_{Q\cap S_n}(u^j)$
\ENDFOR
\STATE Let $i^*\in \displaystyle \arg\min_{j=1,\ldots,p} \inner{u^j}{v^j}$
\RETURN $\{v^1,\ldots, v^{p}\}$ and angle $\arccos \inner{u^{i^*}}{v^{i^*}}$ approximating $\Theta(P,Q)$ 
  \end{algorithmic}
\end{algorithm}

Observe that the vectors $v^j \in \proj_{Q\cap S_j}(g_j)$ can be used to set up an initial cutting-plane for Algorithm~\ref{CP}.
We highlight that a solution of~\eqref{minsupport} need not be a generator of $P$, thus this strategy is a mere heuristic for computing the angle $\Theta(P,Q)$.

\section{Numerical experiments}\label{sec:numexp}
We continue our focus on polyhedral cones, organizing this section into two subsections based on the classes of problems considered. The first subsection presents preliminary results from applying Algorithm~\ref{CP} to randomly generated problems of small dimensions, with the aim of comparing the global and local variants of the algorithm. The second subsection demonstrates the application of our local-solution strategy to an image-set classification problem. In both cases, and for both algorithmic variants, Algorithm~\ref{CP} is warm-started using the heuristic described in Algorithm~\ref{Heur}.

The numerical experiments were conducted in \texttt{MATLAB} on a personal computer: 12th Gen Intel(R) Core(TM) i9 clocked at 2.5 GHz (64 GB RAM). For the global variant, we used \gu~(version 10.0) \cite{Gurobi} to solve the master program~\eqref{master}, while the local-solution variant of Algorithm~\ref{CP} employed \texttt{MATLAB}’s \texttt{fmincon} with default parameters. To evaluate the support function $F_Q$ via item (e) of Proposition~\ref{prop FQ}, we have used \gu. Recall that to compute the dissimilarity measure of two cones $P$ and $Q$, we need to compute the two angles $\Theta(P,Q)$ and $\Theta(Q,P)$. The optimization problems yielding these angles were solved in parallel by using the Parallel Computing  Toolbox~\cite{ParallelComputingToolbox} of \matlab.
All test problems and source codes are available upon request.
\subsection{Random generated problems}
We present preliminary results of Algorithm~\ref{CP}, where the cutting-plane method is implemented by considering two strategies. In particular, to solve the master program~\eqref{master} of Algorithm~\ref{CP}, we either solve it up to global optimality by employing \gu~or we solve it locally by \texttt{fmincon} in \matlab. Accordingly, we denote the computed dissimilarity measures by $\Dis_r^{glob}$ and  $\Dis_r^{loc}$, respectively.
Since the problem yielding these dissimilarity measures is of the $\max$-$\min$ type, with the $\min$ operation always carried out globally, we have that $\Dis_r^{glob} \geq \Dis_r^{loc}$. In both cases, we set the tolerance for the stopping criterion in Algorithm~\ref{CP} to $10^{-3}$ and limited the maximum number of iterations to 200. Additionally, the CPU time allocated for solving the master problem was capped at 300 seconds, while the total CPU time allowed for Algorithm~\ref{CP} was restricted to 1200 seconds (i.e., 20 minutes).

We randomly generate pairs of cones with dimension $p=q=n$, with $n$ varying from $3$ to $25$ and considering $10$ matrices $G$ and $H$ for each dimension. The elements of the matrices $G$ and $H$ are drawn from the normal distribution with zero mean and unitary standard deviation. We set the matrix density to $0.5$ (e.g., half of the elements are set to zero). We use \matlab's \texttt{sprandn} function. For each dimension, we generate 10 different pairs of cones, resulting thus in 230 pairs of polyhedral cones for which we compute the dissimilarity measure.  
\begin{figure}
    \centering
         \subfigure[Dissimilarity measure]{
    \includegraphics[width=0.45\linewidth]{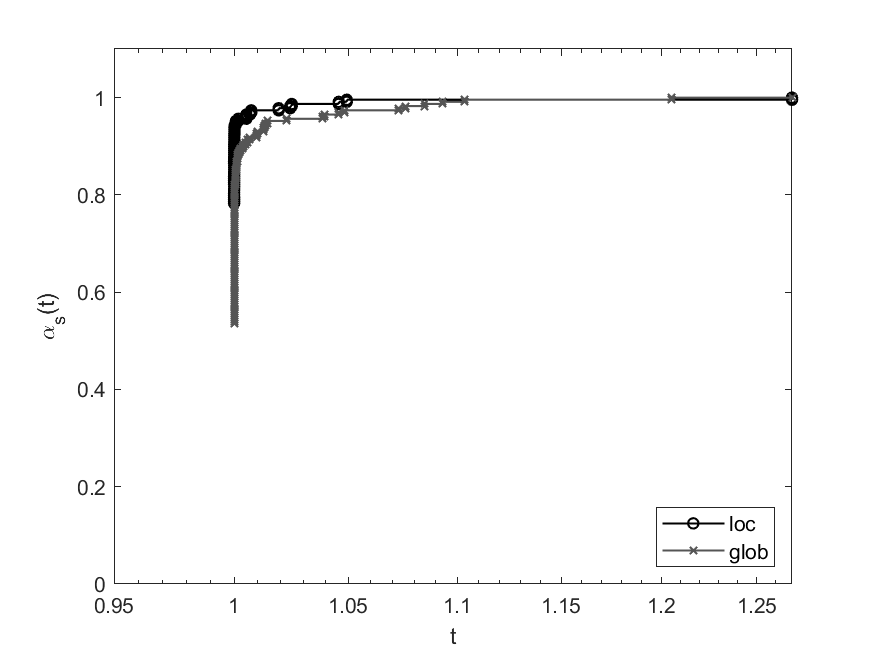}
    }
    \subfigure[CPU time]{
    \includegraphics[width=0.45\linewidth]{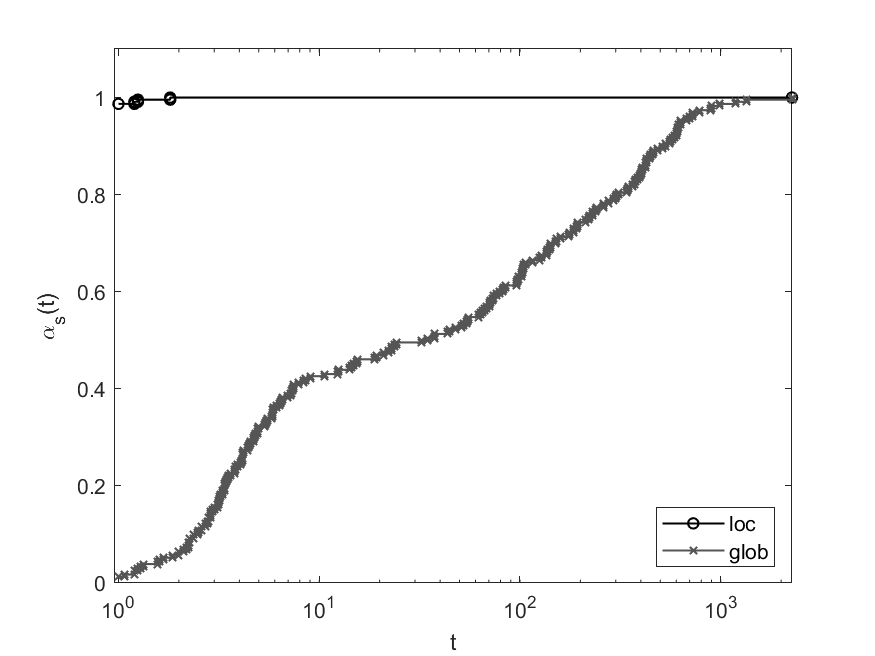}
    }
    \caption{Performance profile of Algorithm~\ref{CP} over 230 test instances, implementing a local or global strategy.}
    \label{fig:pp_alg1_tp1}
\end{figure}

\bigskip

\begin{figure}
    \centering
    \includegraphics[width=0.5\linewidth]{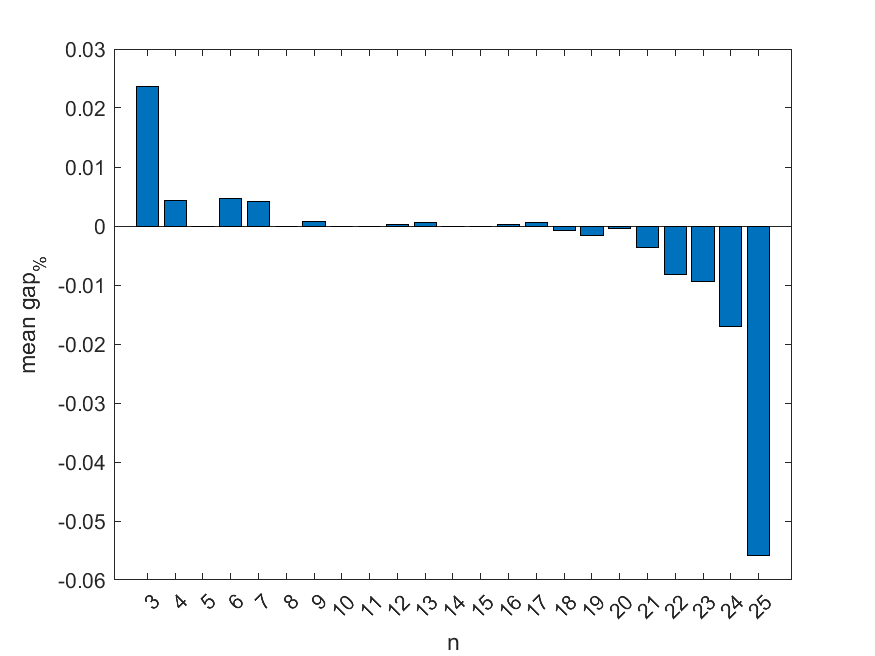}
    \caption{Mean relative gap per dimension over 230 test instances. }
    \label{fig:gap_alg1_tp1}
\end{figure}

Fig.~\ref{fig:pp_alg1_tp1} reports the performance profile~\cite{DoMo} of Algorithm~\ref{CP} over 230 test instances when implementing a local or global strategy for solving the master problem. For each strategy, we plot the proportion of problems for which the best value of the dissimilarity measure in~\eqref{DisrTheta} with $r=2$ was computed by one of the two approaches (leftmost)\footnote{Since smaller metric values indicate better performance in a performance profile, we use the inverse of the computed dissimilarity measure as the performance metric in Fig.~\ref{fig:pp_alg1_tp1}(a).}
 and that was solved within a factor of the best CPU time (rightmost). The two figures clearly show that the two strategies yield approximately the same value for the dissimilarity measure. However, the global one requires a much bigger effort in terms of computational time. For approximately $20\%$ of the instances, implementing the global strategy in Algorithm~\ref{CP} results in the maximum allowed CPU time being reached. The average CPU time required to solve the $230$ instances is $295$ seconds for the global approach, versus $1.2$ seconds for the local approach.  As a further analysis of the numerical results, in Fig.~\ref{fig:gap_alg1_tp1}, we show the mean relative gap per dimension computed considering the values of the dissimilarity measure in~\eqref{DisrTheta} with $r=2$. In particular, we compute
\begin{equation*}
    gap_{\%} = \frac{Dis^{glob}_2(P,Q) - Dis^{loc}_2(P,Q)}{Dis^{glob}_2(P,Q)}
\end{equation*}
for the $10$ instances of dimension $n$. The bar value is the mean over the $10$ considered instances. We can see that for the large instances ($n \geq 20$), the local strategy allows computing a better estimation of the dissimilarity measure. This is due to the fact that in several instances, \gu~reached the time limit of 20 minutes.

For further analysis, we also generate the elements of the matrices $G$ and $H$ as sampled from the uniform distribution in the interval $[0,1]$. We set the matrix density to $0.5$ and use \matlab's \texttt{sprand} function. The results, not reported in this work, show a close performance in terms of the quality of the solution for the two strategies (maximum mean relative gap of $0.01 \%$). However, as before, the global strategy requires significantly more computational effort. 

As a conclusion to these preliminary experiments, we suggest using the local strategy, particularly for problems with a large dimension. We present a practical application in the following subsection.  

\subsubsection{Image-set classification}\label{sec:class}
We consider the ETH-80 dataset~\cite{Chen2020}, which contains object images from 8 different classes, including apples, cars, cows, cups, dogs, horses, pears, and tomatoes. For each class, there are 10 different object instances (e.g., 10 different types of apples) with 41 images captured from different viewpoints. The image resolution is $256 \times 256$ pixels.

In Fig.~\ref{fig:10clas}, we display one image per object class from the same viewpoint, whereas Fig.~\ref{fig:10obj} shows the ten different object instances belonging to the class `apple'. Again, we show the images from the same viewpoint. On the other hand, Fig.~\ref{fig:diffvp} shows five among the 41 viewpoints for one object instance in the classes `cow' and `horse'.

\begin{figure}[H]
\centering
\subfigure[apple]{
\includegraphics[scale=0.22]{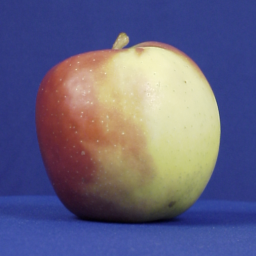}}
\subfigure[car]{
\includegraphics[scale=0.22]{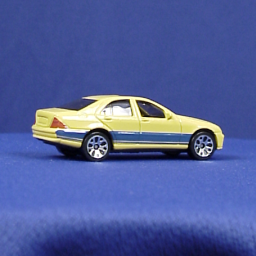}}
\subfigure[cow]{
\includegraphics[scale=0.22]{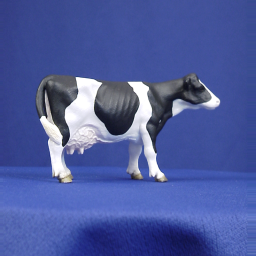}}
\subfigure[cups]{
\includegraphics[scale=0.22]{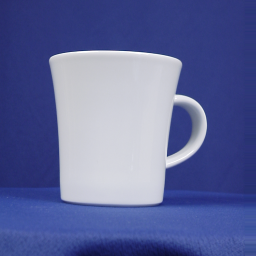}}\\
\subfigure[dog]{
\includegraphics[scale=0.22]{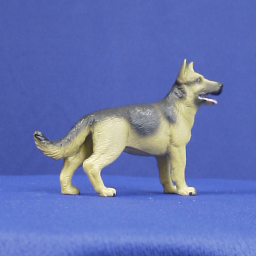}}
\subfigure[horse]{
\includegraphics[scale=0.22]{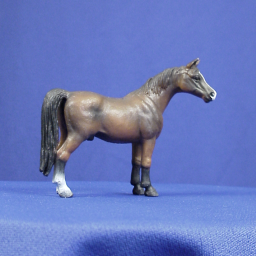}}
\subfigure[pear]{
\includegraphics[scale=0.22]{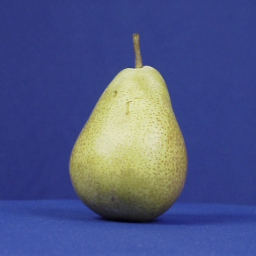}}
\subfigure[tomato]{
\includegraphics[scale=0.22]{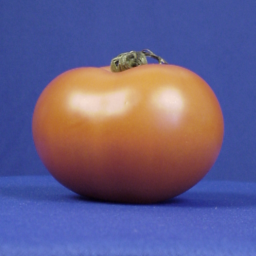}}
\caption{One image per one object in each class.}\label{fig:10clas}
\end{figure}

\begin{figure}[H]
\centering
\subfigure[Object 1]{
\includegraphics[scale=0.2]{figures/apple1-090-338}
}
\subfigure[Object 2]{
\includegraphics[scale=0.2]{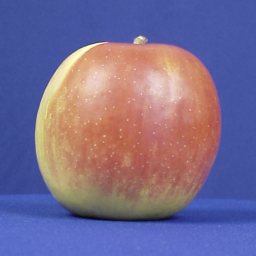}
}
\subfigure[Object 3]{
\includegraphics[scale=0.2]{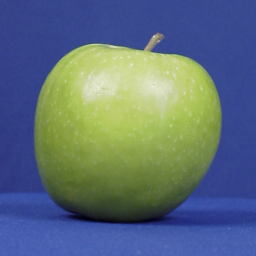}
}
\subfigure[Object 4]{
\includegraphics[scale=0.2]{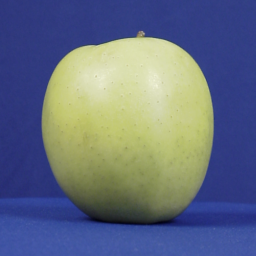}
}
\subfigure[Object 5]{
\includegraphics[scale=0.2]{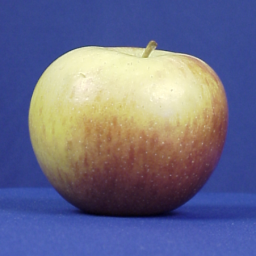}
}
\\
\subfigure[Object 6]{
\includegraphics[scale=0.2]{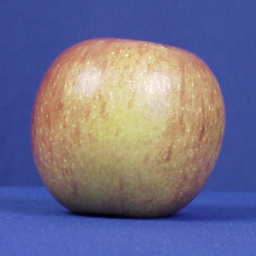}
}
\subfigure[Object 7]{
\includegraphics[scale=0.2]{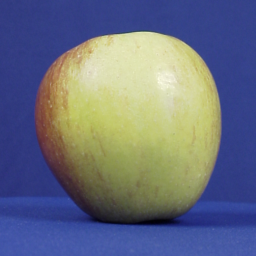}
}
\subfigure[Object 8]{
\includegraphics[scale=0.2]{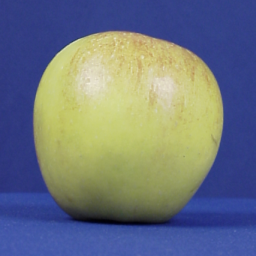}
}
\subfigure[Object 9]{
\includegraphics[scale=0.2]{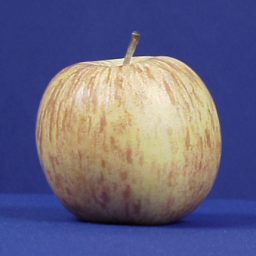}
}
\subfigure[Object 10]{
\includegraphics[scale=0.2]{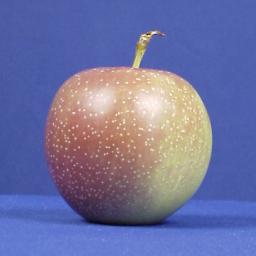}
}
\caption{One image per $10$ objects in one class.}\label{fig:10obj}
\end{figure}

\begin{figure}[H]
\centering
\subfigure[View point 1]{
\includegraphics[width=0.15\textwidth]{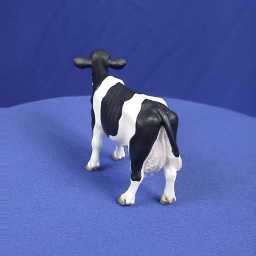}
}
\subfigure[View point 2]{
\includegraphics[width=0.15\textwidth]{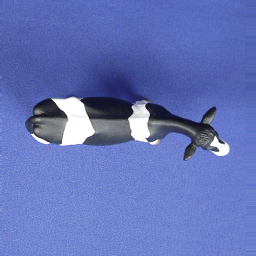}
}
\subfigure[View point 3]{
\includegraphics[width=0.15\textwidth]{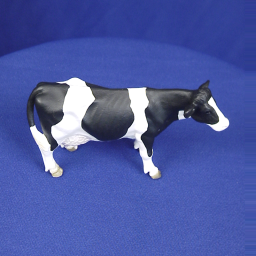}
}
\subfigure[View point 4]{
\includegraphics[width=0.15\textwidth]{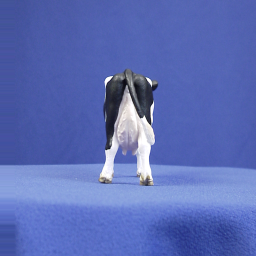}
}
\subfigure[View point 5]{
\includegraphics[width=0.15\textwidth]{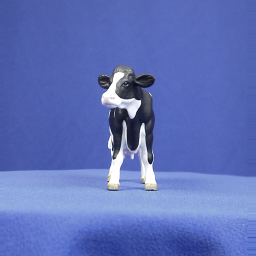}
}
\\
\subfigure[View point 1]{
\includegraphics[width=0.15\textwidth]{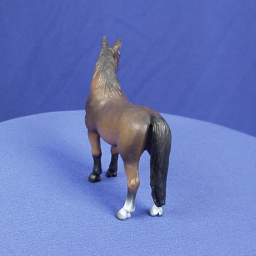}
}
\subfigure[View point 2]{
\includegraphics[width=0.15\textwidth]{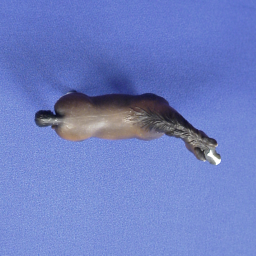}
}
\subfigure[View point 3]{
\includegraphics[width=0.15\textwidth]{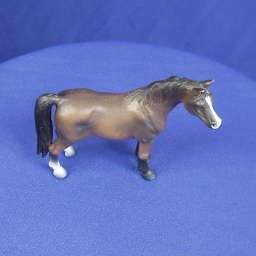}
}
\subfigure[View point 4]{
\includegraphics[width=0.15\textwidth]{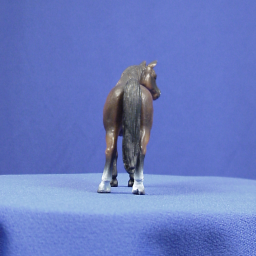}
}
\subfigure[View point 5]{
\includegraphics[width=0.15\textwidth]{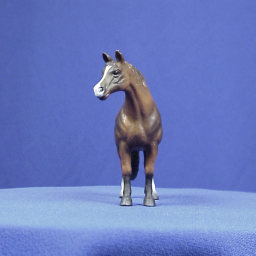}
}
\caption{Five images per one object in the cow and horse classes.}\label{fig:diffvp}
\end{figure}

We represent a set of images as a polyhedral cone. 
As defined in Section 3, to generate the cones $P$ and $Q$, we need to define the matrices $G$ and $H$, respectively. In the following experiments, a column of $G$ (or $H$) is a feature extracted from one image in a given set. We use Convolutional Neural Networks~\cite{Goodfellow2016} to extract features from the images. We use the pretrained deep network ResNet-18 from the Neural Network Toolbox of \matlab~\cite{MATLAB}, which constructs a hierarchical representation of input images, with deeper layers containing higher-level features built upon the lower-level features of earlier layers. We get the CNN features as output of the global pooling layer, `pool5', at the end of the network. The size of the CNN features is $512$.

Once the features are extracted for all the images, we implement the following setup for our experiments. For each class, we randomly select one object instance (i.e., 41 multi-view images) as a test image set. We indicate with $P_i$ the test cone for class $i$,  with $i = 1,2,...,8$. Each cone $P_i$ is generated by considering a matrix $G_i \in \R^{512 \times 41}$, whose columns are the CNN features generated for each of the $41$ images in the set. 

We indicate with $Q^k_j$, the training cone for class $j$ built by randomly drawing $k$ objects among the nine remaining ones. We have $j=1,...,8$ and consider $k = 1,...,5$. For each increase of $k$, we add one randomly chosen object to the current one in the training cone. The goal is to evaluate whether the classification accuracy improves when more information is given to build a cone representing a particular class. Each cone $Q^k_j$ is generated by considering a matrix $H^k_j \in \R^{512 \times 41*k}$, whose columns are generated as above but, in this case, considering $k$ number of objects, with $k=1,...,5$. 

For a fixed $i = i^0$ and $ k = k^0$, we compute $\textrm{Dis}_r(P_{i^0},Q^{k^0}_j)$ for each $j = 1,...,8$, and we classify $P_i$ as belonging to class $j^*$, where $j^* = \textrm{argmin}_{j=1,...8} \textrm{Dis}_r(P_{i^0},Q^{k^0}_j)$. If $j^*= i^0$, then the classification is successful.
The classification accuracy is computed as the number of cones $P_i$ correctly classified over the total number (i.e., 8).

Since the dataset includes $10$ objects per class, we can repeat the above procedure $10$ times, each time considering a different object as the test cone $P_i$ and building the training cone on the remaining objects as explained above. These $10$ different experiments are graphically represented in Fig.~\ref{fig:exp}.
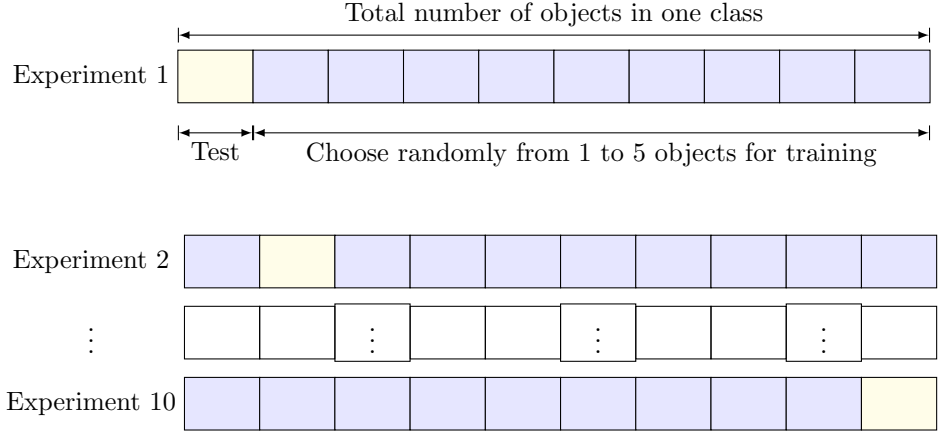
\begin{figure}[t]
\centering 
{\scriptsize
\begin{tikzpicture}
    \matrix (M1) [matrix of nodes, nodes={minimum height=7mm, minimum width=1cm, outer sep=0, anchor=center, draw}, column 1/.style={nodes={draw=none}, minimum width=2cm}, row sep=2mm, column sep=-\pgflinewidth, nodes in empty cells, e/.style={fill=yellow!10}, t/.style={fill=blue!10} ] {
        Experiment 1 & |[e]| & |[t]| & |[t]| & |[t]| & |[t]| & |[t]| & |[t]| & |[t]| & |[t]| & |[t]| \\
    };
    \draw (M1-1-2.north west) ++(0,2mm) coordinate (LT) edge[|<->|, >=latex] node[above]{Total number of objects in one class} (LT-|M1-1-11.north east);
    \draw[|<->|, >=latex] (M1-1-3.south west) ++(0,-4mm) coordinate (LB1) -- (LB1-|M1-1-11.south east) node[midway, below]{Choose randomly from 1 to 5 objects for training};
    \draw[|<->|, >=latex] (M1-1-2.south west) ++(0,-4mm) coordinate (LB2) -- (LB2-|M1-1-2.south east) node[midway, below]{Test};
    \matrix (M2) [matrix of nodes, nodes={minimum height=7mm, minimum width=1cm, outer sep=0, anchor=center, draw}, column 1/.style={nodes={draw=none}, minimum width=2cm}, row sep=2mm, column sep=-\pgflinewidth, nodes in empty cells, e/.style={fill=yellow!10}, t/.style={fill=blue!10}, below=of M1, yshift=-0.5cm] {
        Experiment 2 & |[t]| & |[e]| & |[t]| & |[t]| & |[t]| & |[t]| & |[t]| & |[t]| & |[t]| & |[t]| \\
      \vdots    & & & \node {\vdots}; & & & \node {\vdots}; & & & \node {\vdots}; & \\
        Experiment 10 & |[t]| & |[t]| & |[t]| & |[t]| & |[t]| & |[t]| & |[t]| & |[t]| & |[t]| & |[e]| \\
    };
\end{tikzpicture}
}
\caption{Set up of the numerical experiments.}\label{fig:exp}
\end{figure}

The performance of the proposed approach for image set classification is presented in Fig.~\ref{fig:CNN_Dis}, which reports the accuracy across the ten experiments considering the dissimilarity measures $\Dis_r(P,Q)$ with $r = 1, 2,$ and $\infty$. 
The horizontal axis represents the number of objects used to build the training cone for each class. These figures show that our approach (Algorithm~\ref{CP} with iterates defined by applying \texttt{fmicon} to the master program~\eqref{master}) could correctly classify $7$ classes out of $8$ on average for all the training cone configurations, that is, varying $k$ from $1$ to $5$. The performance is slightly better with a number of objects equal to $4$ or $5$, as more experiments correctly classified all classes. Regarding the use of a particular norm in the definition of the dissimilarity measure, we find the results comparable. The use of the $\ell^1$-norm leads to slightly better average results when the maximum number of objects is used, however, the $l^\infty$ norm shows better value when only one object is considered, i.e., in the \emph{Few-Shot Learning} paradigm, where classification must be performed under limited data conditions.

\begin{figure}[t]
\centering
\subfigure[$\Dis_1(P,Q)$]
{
\includegraphics[scale=0.26]{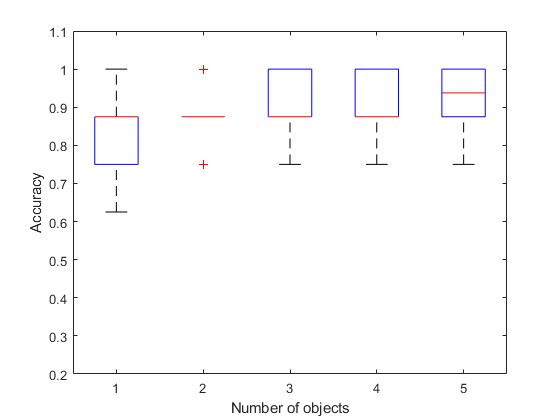}
\label{fig:CNN_a}
}
\subfigure [$\Dis_2(P,Q)$]{
\includegraphics[scale=0.26]{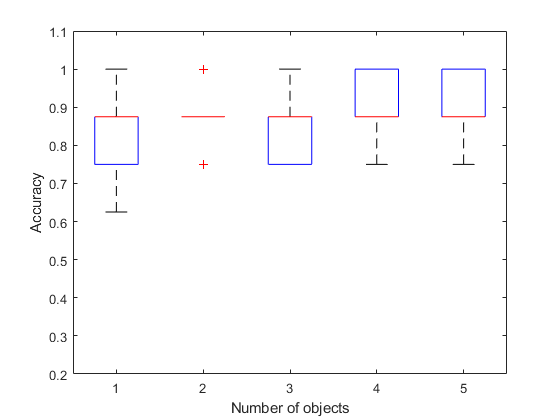}
\label{fig:CNN_aa}
}
\subfigure[$\Dis_\infty(P,Q)$]{
\includegraphics[scale=0.26]{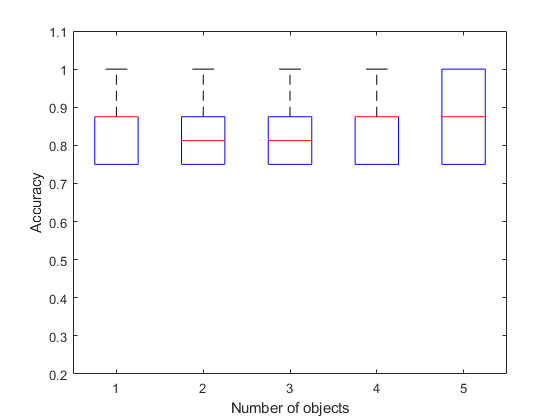}
\label{fig:CNN_aaa}
}
\caption{Performance of the proposed approach using the dissimilarity measures $\Dis_r(P,Q)$ with $r = 1,2 $ and $=\infty$.}\label{fig:CNN_Dis} 
\end{figure}

We also report in Fig.~\ref{fig:CPUtime} the computational burden required for computing the dissimilarity measure. In particular, we compute the CPU time for solving the subproblems to get the angles $\Theta(P_{ i^0},Q^{k^0}_j)$ and $\Theta(Q^{k^0}_j,P_{i^0})$.  
To compute these angles, we employed Algorithm~\ref{CP} under the setting of Theorem~\ref{theo:CP2}. This means that the master problem is only solved to stationarity by using \texttt{fmincon}. Algorithm~\ref{CP} was initialized with the result provided by the heuristic depicted in Algorithm~\ref{Heur}. 
 We compute the mean value of the CPU time necessary to classify $P_i$ for all $i$ and with a fixed $k$.  The computational time to extract the features from the image is not considered here. As expected, the computational time increases as the training cone is built when considering more objects. However, the CPU time is smaller than 30 seconds for the larger case.  

 \begin{figure}[htb]
    \centering
    \includegraphics[width=0.35\linewidth]{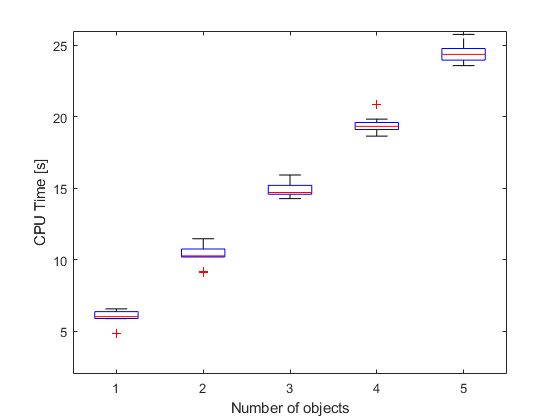}
    \caption{Average CPU time for computing the angles $\Theta(P,Q)$ and $\Theta(Q,P)$. }
    \label{fig:CPUtime}
\end{figure}

It is also interesting to understand which classes are most difficult to classify using our approach. Table~\ref{tab:AccPerClass} shows the percentage of each class that was correctly classified across $10$ experiments, while an increasing number of objects was used to build the training cones. It is evident from these values that distinguishing between cows, dogs, and horses is challenging for our approach. However, the information provided by the additional objects can improve the accuracy for these critical classes. 
\begin{center}
    \begin{table}[htpb]
        \centering
        \scriptsize
        \begin{tabular}{|c|c|c|c|c|c|c|c|c|}
          \hline
             & apple & car & cow & cups & dog & horse & pear & tomato \\ \hline
 $k=1$ & 90  & 100 & 50 & 100 & 70 & 60& 100 & 100\\ \hline
$k=2$  &  100 & 100 & 60 & 100 & 60 & 90 & 100 & 100\\ \hline
$k=3$ & 100 & 100 & 50 & 100  & 50 & 90 & 100 & 100\\ \hline
$k=4$ & 100 & 100 & 60 & 100 & 70 & 90 & 100 & 100\\ \hline
$k=5$ & 100& 100  & 70 & 100 & 70 &90 &100&100\\ \hline
        \end{tabular}
        \caption{Percentage of successful classification per class over the $10$ experiments.}
        \label{tab:AccPerClass}
    \end{table}
\end{center}

\section{Concluding remarks}\label{sec:conclusion}
In this work, we introduced a novel dissimilarity measure between convex cones grounded in the concept of max-min angles, offering a geometric perspective closely related to the Pompeiu-Hausdorff distance. Through a rigorous mathematical and algorithmic framework, we explored several cone configurations where the measure admits simplified or analytic forms. For cones that are linear subspaces of equal dimension, we showed that our measure coincides with the spectral distance, allowing for efficient computation via singular value decomposition. For revolution cones, we derived a closed-form expression for the measure, further illustrating its interpretability and computational tractability in structured settings.

In the more challenging case of polyhedral cones, computing the measure involves minimizing a nonsmooth convex function over a nonconvex feasible set. To face the challenge, we revisited Kelley’s cutting-plane method and proposed a local-solution variant that leverages standard nonlinear programming solvers to tackle the difficult nonconvex master program. 
Our theoretical developments were complemented by a convergence analysis and extensive numerical experiments, including a practical application to image-set classification under few-shot learning conditions. The results highlight the effectiveness and practicality of our approach, with the proposed measure successfully distinguishing between image classes in data-scarce scenarios. In particular, for the considered image-set classification problem, our approach could correctly classify
7 classes out of 8 on average for all the training cone configurations.

Overall, this work contributes a robust and versatile tool for comparing convex cones, with potential implications for geometric data analysis, computer vision, and machine learning tasks where cone-based representations are relevant. 
Future research may explore further global and local optimization strategies for computing the proposed dissimilarity measure in structured settings. In particular, extending the current numerical framework to handle broader classes of cones and designing specialized algorithms could enhance both accuracy and scalability. Furthermore, given the flexibility and interpretability of cone comparisons, our measure lays a solid foundation for further exploration and opens promising avenues for applications in machine learning, especially in tasks involving geometric representations, few-shot learning, and structured data analysis.

\end{document}